\documentclass[a4paper, 10pt, twoside]{article}
\usepackage[utf8]{inputenc}
\usepackage{amsthm}
\usepackage{amsmath}

\usepackage{graphicx} 
\usepackage{hyperref} 
\usepackage{dsfont}
\usepackage[english]{babel}
\usepackage{amsfonts}
\usepackage{enumerate}
\usepackage{enumitem}
\usepackage{amsthm}
\usepackage{leftidx} 
\usepackage{color}
\usepackage{marginnote}

\newcommand{\abs}[1]{\left|#1\right|}
\newcommand{\eps}{\varepsilon}

\renewcommand{\d}{\mathrm{d}}
\renewcommand{\P}{\mathbb{P}}
\newcommand{\E}{\mathbb{E}}
\newcommand{\Ind}[1]{\mathds{1}_{#1}}

\newcommand{\Om}{\Omega}
\newcommand{\vphi}{\varphi}

\newcommand{\epsi}{{\eps_1}}
\newcommand{\epsii}{{\eps_2}}

\newcommand{\tmu}{\tilde{\mu}}
\newcommand{\mc}[1]{\mathcal{#1}}

\newcommand{\cC}{\mc{C}}
\newcommand{\cD}{\mc{D}}

\newcommand{\cF}{\mc{F}}

\newcommand{\cM}{\mc{M}}

\newcommand{\cO}{\mc{O}}

\newcommand{\norm}[1]{\left\lVert #1\right\rVert}

\renewcommand{\sp}[2]{\left\langle #1,#2 \right\rangle}
\newcommand{\tvnorm}[1]{\norm{#1}_{TV}}
\newcommand{\Hsob}{{H^1_0}}
\newcommand{\Hd}{{H^{-1}}}

\newcommand{\dup}[4]{{}_{#1}{\left\langle #2, #3 \right\rangle}_{#4}}

\newcommand{\tow}{\rightharpoonup}
\newcommand{\tows}{\overset{\ast}{\rightharpoonup}}

\newcommand{\N}{\mathbb{N}}

\newcommand{\R}{\mathbb{R}}

\newcommand{\supp}{\mathrm{supp}}
\newcommand{\ie}{i.\,e.\,}
\newcommand{\eg}{e.\,g.\,}
\newcommand{\cf}{cf.\ }

\newcommand{\ext}[1]{#1_\mathrm{ext}}
\newcommand{\dist}{\mathrm{dist}}

\newtheorem{thm}{Theorem}[section]
\newtheorem{prop}[thm]{Proposition}
\newtheorem{lem}[thm]{Lemma}
\newtheorem{Cor}[thm]{Corollary}

\theoremstyle{remark}
\newtheorem{Rem}[thm]{Remark}

\theoremstyle{definition}
\newtheorem{Not}[thm]{Notations}
\newtheorem{Def}[thm]{Definition}
\newtheorem{Ass}[thm]{Assumptions}

\newcommand{\sgn}{\mathrm{sgn}}
\newcommand{\psiexh}{\xi}
\newcommand{\psirep}{\psi}
\numberwithin{equation}{section}
\newcommand{\confn}{\beta}

\usepackage{mathrsfs}
\usepackage{latexsym}
\usepackage{mhequ}
\usepackage{verbatim}
\usepackage{xcolor}
\usepackage{tikz}
\usetikzlibrary{calc,intersections,through,backgrounds}
\usetikzlibrary{decorations.pathreplacing}
\usetikzlibrary{shapes}

\usepackage{float}
\usepackage[margin=1in, marginparwidth=.8in]{geometry}
\usepackage{setspace}

\usepackage{fancyhdr}
\usepackage{url}

\usepackage[]{appendix}

\makeatletter
\DeclareRobustCommand\widecheck[1]{{\mathpalette\@widecheck{#1}}}
\def\@widecheck#1#2{%
    \setbox\z@\hbox{\m@th$#1#2$}%
    \setbox\tw@\hbox{\m@th$#1%
       \widehat{%
          \vrule\@width\z@\@height\ht\z@
          \vrule\@height\z@\@width\wd\z@}$}%
    \dp\tw@-\ht\z@
    \@tempdima\ht\z@ \advance\@tempdima2\ht\tw@ \divide\@tempdima\thr@@
    \setbox\tw@\hbox{%
       \raise\@tempdima\hbox{\scalebox{1}[-1]{\lower\@tempdima\box
\tw@}}}%
    {\ooalign{\box\tw@ \cr \box\z@}}}
  \makeatother

\fancypagestyle{plain}{

  \fancyhf{}
  \fancyfoot[L]{\footnotesize \today}
  \fancyfoot[C]{\thepage}
}
\title{Well-posedness of SVI solutions to singular-degenerate stochastic
  porous media equations arising
  in self-organised criticality}
\author{Marius Neuß}
\date{}

\begin{document}
\maketitle
\begin{abstract}
  We consider a class of generalised stochastic porous media equations with
  multiplicative Lip\-schitz continuous noise. These equations can be related
  to physical models exhibiting self-organised criticality. We show that
  these SPDEs have unique SVI solutions which depend continuously on the
  initial value. In order to formulate this notion of solution and to
  prove uniqueness in the case of a slowly growing nonlinearity, the
  arising energy functional is analysed in detail.


\end{abstract}

\section{Introduction}
We consider a class of singular-degenerate generalised stochastic porous
media equations
\begin{align} \label{eq:1}
  \begin{split}
    \d X_{t}&\in\ \Delta\left(\phi(X_{t})\right)\d t+B(t,X_{t})\d W_{t},\\
    X_0 &=\  x_0,
  \end{split}
\end{align}
on a bounded, smooth domain $\cO\subseteq \R^d$ with zero Dirichlet
boundary conditions and $x_0\in\Hd$, where $\Hd$ is the dual of
$\Hsob(\cO)$. In the following, $W$ is a cylindrical Wiener process
on some separable Hilbert space $U$, and the diffusion coefficients
$B: [0,T]\times \Hd\times \Om \to L_2(U,\Hd)$ take values in the space of
Hilbert-Schmidt operators $L_2(U,\Hd)$. The nonlinearity $\phi: \R\to 2^\R$ is the subdifferential of
a convex lower-semicontinuous symmetric function $\psirep: \R\to\R$ (sometimes
called ``potential''), which grows at least linearly and at most
quadratically for $\abs{x}\to\infty$. As paradigmatic examples, we mention
the maximal monotone extensions of
\begin{equation}\label{eq:55}
  \phi_1(x) = \sgn(x)\left(1 - \Ind{(-1,1)}(x)\right) \quad\text{and}\quad \phi_2(x) = x\left(1 - \Ind{(-1,1)}(x)\right),
\end{equation}
which are encountered in the context of self-organised criticality. Indeed,
equation \eqref{eq:1} with the first nonlinearity in \eqref{eq:55} is
related to a particle model which was first introduced by Bak, Tang and
Wiesenfeld in their celebrated works \cite{BTW} and \cite{BTW88}. We refer
to Section \ref{sec:Motivation} below for details and references.

The main merits of this article are as follows. First, we give a meaning to
\eqref{eq:1} by defining a suitable notion of solution and proving the
existence and uniqueness of such solutions. Second, we extend the
applicability of the framework of SVI solutions, which features several
properties which are desirable independently of the specific equation
presented above. For instance, it applies to stochastic partial
differential equations (SPDE) with a very general nonlinear
drift term, and solutions for general initial data can be identified by
means of the equation and not only in a limiting sense.

We briefly outline the strategy that we are going to apply. First, we rewrite
\eqref{eq:1} into the form
\begin{equation}\label{eq:91}
  \d X_t \in -\partial\vphi(X_t)\,\d t + B(t, X_t)\,\d W_t,
\end{equation}
which incorporates the multivalued function $\phi$ into an energy
functional $\vphi: \Hd \to [0,\infty]$. For example, in case of the nonlinearity
$\phi_1$ in \eqref{eq:53}, we define
\begin{equation}\label{eq:90}
  \vphi(u)=
    \begin{cases}
      \norm{\psirep(u)}_{TV},\quad & \text{if }u\text{ is a finite Radon
        measure on }\cO,\\
      +\infty, & \text{else,}
    \end{cases}
\end{equation}
where $\psi$ is the anti-derivative of $\phi$, \ie $\partial\psi =
\phi$, with $\psi(0) = 0$. For the precise definition of a convex function of a measure, we
refer to Section \ref{sec:preparatory} below. We then derive a stochastic
variational inequality (SVI) from \eqref{eq:91} and define a corresponding
notion of solution, see Definition \ref{Def-SVI} below. In order to
construct such a solution we first show that $\vphi$ as defined above is
lower-semicontinuous, which then allows to show the convergence of an
approximating sequence gained by a Yosida approximation of the nonlinearity
and the addition of a viscosity term. Furthermore, in the proof of
uniqueness, it is crucial to show that $\vphi$ can be well approximated by
its values on $L^2$, which we ensure by showing that it coincides with the
lower-semicontinuous hull of $\vphi|_{L^2}$ in $\Hd$. To this end, we will
construct approximating sequences by an interplay of mollification and
shifts, inspired by the construction of \cite[Lemma A6.7]{Alt}. This
constitutes one technical focus of this work.

The structure of this article is as follows: In the subsequent sections of
the introduction, we will give a brief overview on the mathematical
literature concerning the solution theory of generalised stochastic porous
media equations, and we will point out how equation \eqref{eq:1} is
motivated by the physics literature. In Section \ref{sec:requirements} we
state the precise assumptions and formulate the first main result of this
article, in which the well-posedness of Equation \eqref{eq:1} is
established (see Theorem \ref{main-thm} below). We prove the
lower-semicontinuity of the abovementioned energy functional $\vphi$ and
the property of $\vphi$ being the lower-semicontinuous hull of
$\vphi|_{L^2}$ in $\Hd$ in Section \ref{sec:preparatory}, the latter of
which is the second main result (see Theorem \ref{approx-thm} below). In
Section \ref{sec:proof-existence}, the well-posedness result will be
proved, following the arguments of \cite[Section 2]{G-R}.

\subsection{Mathematical Literature}
In the recent decades, stochastic porous media equations have been very
present in the mathematical literature. For the original case
\begin{equation}\label{eq:60}
  \d X_{t} = \Delta \phi(X_{t})\d t+B(t,X_{t})\d W_{t},
\end{equation}
where $\phi(r) = r^{[m]} := \abs{r}^{m-1}r$ for $r\in\R$ and $m\geq 1$
($m=1$ representing the stochastic heat equation), a concisely summarised
well-posedness analysis can be found in \cite{Roeckner},
which goes back to the work of Krylov and Rozovskii \cite{Krylov-Roz} and
Pardoux \cite{Pardoux}. In \cite{RRW}, the theory is extended to the fast
diffusion case $m\in(0,1)$, and other nonlinear functions $\phi$ are
considered. A setting with a more general monotone and differentiable
nonlinearity is considered in \cite{BDPrR-existence-nonneg}.

A severe additional difficulty arises when one considers the limit case
$m=0$, in which $\phi$ becomes multivalued. The first articles treating this
type of porous medium equations, \cite{BDPrR-existence-strong} and
\cite{BDPrR-SPMEandSOC}, either require $\phi$ to be surjective or more
restrictions on the initial state or the noise. In \cite{G-T-multivalued},
the $m=0$ limit of \eqref{eq:60} can be treated, but one has to restrict
to more regular initial data or to the concept of limiting solutions. For
general initial conditions, this notion of solution contains no
characterisation in terms of the equation, which is often necessary for
further work such as stability results (see \eg \cite{G-T-stability}).

In \cite{BDPrR-Diffusivity} and later in \cite{B-R-SVITVF,G-R-plaplace},
the concept of stochastic variational inequalities (SVIs) and a
corresponding notion of solution have been used to overcome these
issues. We note that in \cite{G-R-plaplace}, an identification of a
functional as a lower-semicontinuous hull was needed in the context of
$p$-Laplace type equations with a $C^2$ potential, going back to results
from \cite{Anzellotti85,Ferro}. In \cite{G-R}, the existence and uniqueness
of SVI solutions was proven for the $m=0$ limit of \eqref{eq:60}, for which
a refinement of previous methods became necessary, because the naive choice
for the energy functional does not lead to an energy space with adequate
compactness properties. The arising difficulties when setting up the energy
functional are similar to the ones mentioned above for $\vphi$ from
\eqref{eq:90}. They have been overcome in \cite{G-R} by using the specific
shape of the nonlinearity, which allows to set the energy functional to
\begin{displaymath}
  \vphi(u)=
    \begin{cases}
      \norm{u}_{TV},\quad & \text{if }u\text{ is a finite Radon
        measure on }\cO,\\
      +\infty, & \text{else}
    \end{cases}
\end{displaymath}
for $u\in\Hd$, which then allows to use structural properties of the TV
norm. With more regularity or structural assumptions on the noise and/or
the initial state, more regularity for SVI solutions or the existence of
strong solutions can be proved, as \eg in
\cite{G-R,G-R-plaplace,B-R-SVITVF,Gess-strongsolns}. For the regularisation
by noise of quasi-linear SPDE with possibly singular drift terms, we also
mention the works \cite{Gassiat-Gess,Gyongy-Pardoux}.

We next mention several different approaches to stochastic porous media
equations. The article \cite{BRR-Rd} considers the equation on an unbounded
domain, the works \cite{Barbu-Bogachev, DPrR-weaksolns} use an approach via
Kolmogorov equations. In \cite{BR-operatorial}, an operatorial approach to
SPDE is introduced which can be applied to generalised stochastic porous
media equations with continuous nonlinearities. In \cite{G-H,Deb-Hof-Vov} and
\cite{DGG}, stochastic porous media equations are solved in the sense of kinetic
or entropy solutions, respectively. Previous works in those directions are,
e.\,g., \cite{Bauzet-Vallet-Wittbold15,Deb-Vovelle} and \cite{Biswas-Majee,Feng-Nualart,Kim03}. \cite{GS17} makes use of a rough
path approach leading to pathwise rough kinetic/entropy solutions and
including regularity results, with
\cite{Friz-Gess,LPS13} as some of the related preceding works.

Regarding the construction and analysis of the energy functional arising in
the context of SVIs, we rely on techniques from
\cite{Temam-Demengel,Temam-book} on convex functionals of Radon
measures. For the deterministic theory on porous medium equations, we refer
to \cite{OttoPME} and \cite{VazquezPME}. Regarding results on the long-time
behaviour of singular-degenerate SPDE, see \eg \cite{FGS17,Gess-RAsingular}
for the existence of random attractors,
\cite{G-T-ergodicity,G-D-T,Neuss-ergodicity} for ergodicity and
\cite{Gess-FTE,BDPrR-FTE} for finite-time extinction in the case of purely
multiplicative noise.

\subsection{Self-organised criticality (SOC)} \label{sec:Motivation}

The model (1.1) can to some extent be associated with processes exhibiting
self-organised criticality (SOC). This concept postulates that many
randomly driven processes featuring a critical threshold, at which
relaxation events are triggered, possess a non-equilibrium statistical
invariant state, in which intermittent events can be observed, the size of
which is distributed by a power law. SOC has been initially discussed in
view of certain cellular automaton models, which are introduced and
explained in much detail in \cite{BTW} and \cite{BTW88}, as well as later
by \cite{Arenas}. In these models, particles can be interpreted as units of
granular material piling up, which coined the notion of ``sandpile
models''. Other applications, where self-organised critical behaviour has
been observed, are the size of landslides \cite{Noever}, earthquakes (the
famous Gutenberg-Richter law, see \cite{Gutenberg-Richter}) and stock
prices \cite{Mandelbrot}.

In \cite{Diaz-G} and \cite{Diaz-G-PhysRevA}, the abovementioned sandpile
models are related to a model similar to \eqref{eq:1}, \ie a stochastic
process in a continuous function space where mass of a continuously
distributed size is both added and subtracted. In contrast to the
assumptions mentioned above, the potential in \cite{Diaz-G} is only
one-sided. As this leads to a process just forced towards $-\infty$, where
no avalanches would occur, we consider symmetric potentials instead.

The underlying mechanisms of SOC have been a matter of lively discussion in
the literature, see \eg \cite{Watkins-Pruessner} for a review. The present
work is supposed to contribute to this question by noting that SPDEs with
singular-degenerate drift and additive noise incorporate several
characteristic properties of the original sandpile models, such as deterministic
dynamics which are locally switched on at a certain threshold. However,
they also differ from them in other perspectives, such as the non-discrete
structure. By setting up a theory for those processes, we ultimately hope
to gain insight into their long-time statistics, see \eg
\cite{Neuss-ergodicity}. Thereby, we aim to investigate whether SOC extends
to the continuous setting and potentially set the stage for new ways of
explaining this statistical effect.

\subsection{General notation} \label{sec:general-notation} Unless specified
differently, function or measure spaces will be understood to be defined on
a smooth, bounded domain $\cO\subset\R^d, d\in\N$. We write
$L^p = L^p(\cO)$ for the usual Lebesgue spaces with norm
$\norm{\cdot}_{L^p}$ and scalar product $\sp{\cdot}{\cdot}_{L^2}$ if
$p=2$. The Lebesgue measure is denoted by $\d x$, and a measure with
density $h\in L^1$ with respect to $\d x$ is denoted by
$h\,\d x$. Furthermore, $\Hsob = \Hsob(\cO)$ denotes the Sobolev space of
$L^2$ functions whose first-order weak derivatives exist and are in $L^2$,
and which have zero trace, with norm
$\norm{u}_\Hsob = \norm{\nabla u}_{L^2}$. The full space analogues
$L^2(\R^d)$, $H^1(\R^d)$ are defined correspondingly. Furthermore, let
$\Hd$ denote the topological dual of $\Hsob$. We use $-\Delta$ to denote
the corresponding Riesz isomorphism, which gives rise to the inner product
\begin{displaymath}
  \sp{u}{v}_\Hd = \dup{\Hd}{u}{(-\Delta)^{-1}v}{\Hsob}\quad\text{for
    all }u,v\in\Hd,
\end{displaymath}
where the notation $\dup{V'}{u}{v}{V} = \dup{V}{v}{u}{V'}$ denotes
evaluating a functional $u$ belonging to the dual space $V'$ of a
Banach space $V$ at a vector $v\in V$.

Moreover, we let $\cC^0_0 = \cC^0_0(\cO)$ denote
the set of all continuous functions on $\cO$ vanishing at the
boundary, while we write $\cC^0_c = \cC^0_c(\cO)$ for continuous functions with compact
support. The same notation applies to spaces $\cC^k$ of $k$ times
continuously differentiable functions.

For $m\in[0,1]$ we define the set
\begin{displaymath}
  L^{m+1}\cap\Hd := \left\{v\in L^{m+1}: \exists\, C\geq0\text{ such
      that}\int v \eta\,\d x \leq C\norm{\eta}_\Hsob \text{ for all }\eta\in C^1_c\right\}.
\end{displaymath}
Note that $L^2 = L^2\cap\Hd$ by the Cauchy-Schwarz and Poincaré
inequalities. To each $v\in L^{m+1}\cap\Hd$ one can injectively assign a map
\begin{equation}\label{eq:98}
  \cC^1_c \ni \eta \mapsto \int v\eta\, \d x.
\end{equation}
By continuity, \eqref{eq:98} can be injectively extended to a bounded
linear functional on $\Hsob$, which we call $\iota_m(v)$. The resulting map
$\iota_m: L^{m+1}\cap\Hd\to\Hd$ is thus injective, which allows to
identify $v$ with $\iota_m(v)$.

Let $\cM = \cM(\cO)$ be the space of all signed Radon measures on $\cO$ with
finite total variation, which is isomorphic to the dual space
$\left(\cC_0^0\right)'$ via
\begin{equation}
  \cM \ni \mu \mapsto \tilde{\mu}\in \left(\cC_0^0\right)',\quad
  \tilde{\mu}(f) = \int f \d\mu.
\end{equation}
This allows us to use $\left(\cC_0^0\right)'$ and $\cM$, as well as $\tilde{\mu}$
and $\mu$ interchangeably. The variation measure of $\mu\in\cM$ is denoted
by $\abs{\mu} := \mu_++ \mu_-$ and the total variation of $\mu$ is given by
\begin{displaymath}
  \norm{\mu}_{TV} = \abs{\mu}(\cO).
\end{displaymath}
Note that the total variation is also the operator norm if the measure
is interpreted as an element of $\left(\cC_0^0\right)'$ by the
Riesz-Markov representation theorem (see e.\ g.\
\cite[Theorem 1.200]{Fonseca}). We define the space of \textit{measures of
  bounded energy} by
\begin{displaymath}
  \cM\cap\Hd := \left\{\mu\in \cM: \exists \,C\geq
               0\text{ such that}\int \eta(x)\,\d\mu(x) \leq  C\norm{\eta}_\Hsob\ \text{for all }\eta\in \cC^1_c(\cO)\right\}.
\end{displaymath}
By a density argument, restricting a measure $\mu\in\cM\cap\Hd$ to a
function on $\cC^1_c$ is an injective operation. Moreover, by continuity
$\mu|_{\cC^1_c}$ can be injectively extended to a bounded linear functional on
$\Hsob$, which we call $\iota(\mu)$. The resulting map
$\iota: \cM\cap\Hd\to\Hd$ is thus injective, which allows to identify $\mu$
with $\iota(\mu)$.

In general, constants may vary from line to line, but are always
positive and finite.

\subsection{Acknowledgements}
I would like to thank my advisor Benjamin Gess for his support throughout
my work on this article. Support by the International Max Planck Research School
in Leipzig and by the ``Cusanuswerk -- Bischöfliche
Studienförderung'' is gratefully acknowledged.

\section{Assumptions and main result} \label{sec:requirements}

\begin{Ass}\label{assumptions}
  We require the following assumptions throughout
  this article.
  \begin{enumerate}[label=\textbf{(A\arabic*)},ref=(A\arabic*)]
  \item \label{item:16} $W$ is a cylindrical $\mathrm{Id}$-Wiener process
    in some separable Hilbert space $U$ defined on a probability space
    $(\Om, \cF, \P)$ with normal filtration $(\cF_t)_{t\geq 0}$, which
    means the following: There is a Hilbert-Schmidt embedding $J$ from $U$
    to another Hilbert space $U_1$, which can be chosen to be bijective
    (see \eg \cite[Remark 2.5.1]{Roeckner}). Defining $Q_1:= JJ^*$, $Q_1$
    is linear, bounded, non-negative definite, symmetric and has finite
    trace, so that we obtain a classical $Q_1$-Wiener process $\tilde{W}$
    on $U_1$. Moreover, for an operator $\tilde{B}: U\to H^{-1}$ we have
    \begin{equation}\label{eq:71}
      \tilde{B}\in L_2(U, H^{-1}) \Leftrightarrow \tilde{B}\circ J^{-1} \in
      L_2\left(Q_1^{\frac1{2}}(U_1), H^{-1}\right),
    \end{equation}
    such that if \eqref{eq:71} is satisfied, we can define
    \begin{displaymath}
      \int_0^T \tilde{B}\, \d W_t := \int_0^T \tilde{B}\circ J^{-1}\d\tilde{W}_t.
    \end{displaymath}
  \item The diffusion coefficients
    $B: [0,T]\times\Hd\times\Om\to L_2(U,\Hd)$ take values in the space of
    Hilbert-Schmidt operators, are progressively measurable and satisfy
    \begin{align}
      \norm{B(t,v)- B(t,w)}_{L_2(U,\Hd)}^2 &\leq C\norm{v-w}_\Hd^2\quad&\text{for all }v,w\in \Hd,\label{eq:22}\\
      \norm{B(t,v)}_{L_2(U,L^2)}^2 &\leq C(1 +
                                     \norm{v}_{L^2}^2)\quad&\text{for all
                                     }v\in L^2, \label{eq:25}\\
      \norm{B(t,0)}_{L_2(U,\Hd)}^2 &\leq C, \label{eq:73}
    \end{align}
    for some constant $C>0$ and all $(t,\omega)\in[0,T]\times\Om$.
  \item \label{item:11} The so-called \textit{potential}
    $\psirep: \R\to [0,\infty)$ is convex and lower-semicontinuous, and we
    assume $\psirep(0) = 0$, which then implies
    $0 \in \partial \psirep(0)$. For simplicity, we furthermore impose the
    symmetry assumption $\psirep(x) = \psirep(-x)$ for all $x\in\R$.
  \item \label{item:12} Define $\phi = \partial \psirep: \R\to 2^\R$, the
    subdifferential of $\psirep$, and assume for all $r\in\R$
    \begin{equation}\label{eq:53}
      \inf\{\abs{\eta}^2:\eta\in\phi(r)\}\leq
      C(1+\abs{r}^{2}).
    \end{equation}
  \end{enumerate}
  In case that
  \begin{equation}\label{eq:67}
    \lim_{\abs{x}\to\infty}\frac{\psirep(x)}{\abs{x}} \to \infty,
  \end{equation}
  \ie $\psirep$ is superlinear, we require
  \begin{enumerate}[label=\textbf{(A\arabic*)},ref=(A\arabic*)]\setcounter{enumi}{4}
  \item \label{item:13} There exists $ m\in(0,1]$, such that
    $\psirep(v) \in L^1(\cO)\quad\text{if and only if}\quad v\in L^{m+1}(\cO)$.
  \end{enumerate}
  In case that the potential is sublinear, \ie that there exists a
  constant $C>0$ such that
  \begin{equation}\label{eq:20}
    \psirep(x)\leq C(1+\abs{x})\quad \text{for all } x\in\R,
  \end{equation}
  we require
  \begin{enumerate}[label=\textbf{(A5')},ref=(A5')]
  \item \label{item:9} 
    There exists $y>0$ such that $\psirep(y)>0$.
  \end{enumerate}
  Note that by convexity, Assumption \ref{item:9} implies that
    \begin{displaymath}
      \psirep(x) \geq \frac{\psirep(y)}{y}\abs{x} - \psirep(y)\quad \text{for all }x\in\R.
    \end{displaymath}
\end{Ass}

Next, we define the energy functional for the notion
of solution we are going to consider.

\begin{Def}\label{def-energy-fnal}
  Let Assumptions \ref{assumptions} be satisfied.
  \begin{enumerate}[label=(\roman*), ref=(\roman*)]
  \item \label{item:1}In the case of a superlinear potential, \ie if \eqref{eq:67} is
    satisfied, we define for $u\in\Hd$ the functional
    \begin{equation}\label{eq:76}
      \vphi(u)=
      \begin{cases}
        \int \psirep(u)\ \d x,\quad & \text{if
        }u\in L^{m+1}\cap
        \Hd,\\
        +\infty, & \text{else,}
      \end{cases}
    \end{equation}
    where $m$ is the exponent from \ref{item:13}.
    \item \label{item:2} In the case of a
    sublinear potential, \ie if \eqref{eq:20} is satisfied, we define for
    $u\in\Hd$ the functional
    \begin{equation}\label{eq:77}
      \vphi(u)=
      \begin{cases}
        \norm{\psirep(u)}_{TV},\quad & \text{if }u\in \cM\cap\Hd,\\
        +\infty, & \text{else,}
      \end{cases}
    \end{equation}
    where the construction of a nonlinear functional of a
    measure, which is needed in \eqref{eq:77}, is given in Definition
    \ref{fofmu} below.
  \end{enumerate}
\end{Def}

\begin{Rem} \label{justification-GF} The choice of the energy functional in
  Definition \ref{def-energy-fnal} allows us to reformulate \eqref{eq:1} as
  a gradient flow, \ie to rewrite it in the form
  \begin{align}\label{eq:68}
    \begin{split}
      \d X_t &\in -\partial \vphi(X_t) \d t + B(t,X_t)\d W_t,\\
      X_0 &= x_0,
    \end{split}
  \end{align}
  where the subdifferential is well-defined due to Proposition
  \ref{conv+lsc} below. More precisely, let a ``classical'' solution to
  \eqref{eq:1} with $x_0\in\Hd$ be defined as an
  $(\cF_t)_{t\geq 0}$-adapted process $X \in L^2(\Om; \cC([0,T];\Hd)$ with
  the following properties: $\P$-almost surely, for all $t\in[0,T]$ we have
  $X_t\in L^2$, there is a choice $v_t\in \phi(X_t)$ such that
  $v_t\in\Hsob$, and
  \begin{displaymath}
    X_t = x_0 + \int_0^t\Delta v_r\, \d r + \int_0^tB(r,X_r)\ \d W_r.
  \end{displaymath}
  Furthermore, we impose $\Delta v \in L^2([0,T]\times\Om;\Hd)$. Then, one
  can check that if $X$ is
  a classical solution in this sense, $(X, \Delta v)$ is a strong solution
  to
  \begin{align}
    \begin{split}
      \d X_t &\in -\partial \vphi(X_t) \d t + B(t,X_t)\d W_t\\
      X_0 &= x_0
    \end{split}
  \end{align}
  in the sense of \cite[Appendix A]{G-R}.
\end{Rem}

Now we are in the position to formulate the notion of solution we want to
consider.
\begin{Def}[SVI solution] \label{Def-SVI} Given Assumptions
  \ref{assumptions}, let $x_0\in L^2(\Om, \cF_0; \Hd)$, $T>0$ and $\vphi$
  be defined
  as in Definition \ref{def-energy-fnal}. We say that an $\cF_t$-adapted
  process $X\in L^2(\Om; \cC([0,T];\Hd))$ is an SVI solution to \eqref{eq:1}
  if the following conditions are satisfied:
  \begin{enumerate}[label=(\roman*)]
  \item \label{item:6}(Regularity)
    \begin{displaymath}
      \vphi(X)\in L^1([0,T]\times\Om).
    \end{displaymath}
  \item \label{item:7}(Variational inequality) For each
    $\cF_t$-progressively measurable process
    $G\in L^2([0,T]\times\Om;\Hd)$, and each $\cF_t$-adapted process
    $Z\in L^2(\Om;\cC([0,T];\Hd)) \cap L^2([0,T]\times\Om;L^2)$ solving the
    equation
    \begin{displaymath}
      Z_t-Z_0 = \int_0^tG_s\,\d s + \int_0^tB(s,Z_s)\,\d W_s\quad\text{for
        all }t\in[0,T],
    \end{displaymath}
    we have
    \begin{align}
      \begin{split}\label{eq:4}
        \E\norm{X_t-Z_t}_\Hd^2 + 2\E\int_0^t\vphi(X_r)\d r
        \ \leq&\ \E\norm{x_0-Z_0}_\Hd^2 + 2\E\int_0^t\vphi(Z_r)\d r\\
        &-2\,\E\int_0^t\sp{G_r}{X_r-Z_r}_\Hd\d r\\
        & +C\,\E\int_0^t\norm{X_r-Z_r}_\Hd^2\d r \quad\text{for
          all }t\in[0,T]
      \end{split}
    \end{align}
    for some $C>0$.
  \end{enumerate}
\end{Def}

\begin{Rem} \label{relaxation-rem} It is shown in \cite[Remark 2.2]{G-R}
  that if $(X,\eta)$ is a strong solution to \eqref{eq:68} in $\Hd$, as defined
  in \cite[Appendix A]{G-R}, then $X$ is an SVI solution to
  (\ref{eq:1}).
\end{Rem}

The main result of this article is as follows.
\begin{thm}\label{main-thm}
  Given Assumptions \ref{assumptions}, let $x_0\in L^2(\Om, \cF_0;\Hd)$ and
  $T>0$. Then there is a unique SVI solution $X$ to \eqref{eq:1}. For two
  SVI solutions $X, Y$ with initial conditions
  $x_0, y_0\in L^2(\Om, \cF_0; \Hd)$, we have
  \begin{equation}\label{eq:19}
    \sup_{t\in[0,T]}\E\norm{X_t-Y_t}^2_\Hd\leq C\,\E\norm{x_0-y_0}^2_\Hd.
  \end{equation}
\end{thm}

The proof of this theorem will be given in Section
\ref{sec:proof-existence} below.

\section{Properties of the energy functional} \label{sec:preparatory}

The aim of this section is to make Definition \ref{def-energy-fnal}
rigorous by recalling the concept of convex functionals on measures, and to
prove certain properties of the energy functional defined in Definition
\ref{def-energy-fnal}, which are needed for the proof of the main
theorem. We start with some basic concepts concerning convex functions.

\begin{Def}\label{Def-conj-rec}
  Let $f: \R\to[0,\infty]$ be a convex and lower-semicontinuous function
  with $f(0) = 0$. We then define its convex conjugate
  $f^*: \R\to[0,\infty]$ by
  \begin{equation}\label{eq:80}
    f^*(x) = \sup_{y\in\R}\,(xy - f(y)),
  \end{equation}
  and its recession function $f_\infty: \R\to[0,\infty]$ by
  \begin{equation} \label{eq:28}
    f_\infty(x) = \lim_{t\to\infty}\frac{f(tx)}{t}.
  \end{equation}
\end{Def}

\begin{Rem}\label{facts-conj-rec}
  Note that $f_\infty$ and $f^*$ are convex. If $f$ is symmetric, so are
  $f_\infty$ and $f^*$. Moreover, $f_\infty$ is positively homogeneous.
\end{Rem}

For the notion of solution that we are aiming at, we need the concept of a
convex function of a measure, which has been developed in \cite{Temam-Demengel}.
\begin{Def}\label{fofmu}
  Let $\psirep$ satisfy \eqref{eq:20} as well as Assumptions
  \ref{assumptions} \ref{item:11}, \ref{item:9}. Define the set
  \begin{displaymath}
    \cD_\psirep = \{ v\in \cC_c^0(\cO): \psirep^\ast(v) \in L^1(\cO)\}
  \end{displaymath}
  and let $\mu\in\cM(\cO)$. We then
  define the positive measure $\psirep(\mu)\in\cM(\cO)$ by
  \begin{equation}\label{eq:81}
    \int_\cO \eta\, \psirep(\mu) := \dup{\cM(\cO)}{\psirep(\mu)}{\eta}{\cC_0^0(\cO)} := \sup\left\{\int_\cO v\eta\ \d\mu - \int_\cO \psirep^\ast(v)\eta\ \d x:
      v\in\cD_\psirep\right\}
  \end{equation}
  for $\eta\in \cC_0^0(\cO), \eta\geq 0$, and for general $\eta\in \cC_0^0(\cO)$ we set
  \begin{displaymath}
    \dup{\cM(\cO)}{\psirep(\mu)}{\eta}{\cC_0^0(\cO)} = \dup{\cM(\cO)}{\psirep(\mu)}{\eta\vee 0}{\cC_0^0(\cO)} -
    \dup{\cM(\cO)}{\psirep(\mu)}{(-\eta)\vee 0}{\cC_0^0(\cO)},
  \end{displaymath}
  according to \cite[Theorem
  1.1]{Temam-Demengel}.
\end{Def}
\begin{Rem} \label{Representation-remark}
  As argued in \cite[Lemma 1.1]{Temam-Demengel}, one can write for $\mu\in\cM(\cO)$
  \begin{displaymath}
    \int_{\cO} \psirep(\mu) = \norm{\psirep(\mu)}_{TV} = \sup\left\{\int_\cO v\,\d\mu - \int_\cO \psirep^*(v)\,\d x:
        v\in\cD_\psirep\right\}.
  \end{displaymath}
\end{Rem}
\begin{Rem}\label{decomp-rem}
  Let $\mu\in\cM(\cO)$ with Lebesgue decomposition $\mu^a + \mu^s$, where
  $\mu^a$ has the density $h\in L^1(\cO)$ with respect to the Lebesgue
  measure. Then, by \cite[Theorem 1.1]{Temam-Demengel}, we have
  \begin{equation}\label{eq:82}
    \int_\cO \eta\ \psirep(\mu) = \int_\cO \eta(x) \psirep(h(x)) \d x + \int_\cO \eta\
    \psirep_\infty(\mu^s),
  \end{equation}
  where the recession function $\psirep_\infty$ is defined as in (\ref{eq:28}). In
  particular, this formulation shows the useful fact that
  \begin{equation} \label{eq:29}
    \psirep(\mu) = \psirep(\mu^a) + \psirep(\mu^s).
  \end{equation}
\end{Rem}

Our next aim is to prove the lower-semicontinuity of the energy functional
defined in Definition \ref{def-energy-fnal} and Definition
\ref{fofmu}. First, we show that the Radon measure $\psirep(\mu)$
constructed in Definition \ref{fofmu} controls the norm of its original
measure $\mu$ in the following way.
\begin{lem} \label{varphi-TV}
  Let $\psirep$ satisfy \eqref{eq:20} as well as
  Assumptions \ref{assumptions} \ref{item:11}, \ref{item:9}. Let $\mu\in\cM(\cO)$ and let $y>0$ such that
  $\psirep(y)>0$ as demanded in Assumption \ref{assumptions} \ref{item:9}. Then
  \begin{displaymath}
    \norm{\psirep(\mu)}_{TV} \geq \frac{\psirep(y)}{y} \norm{\mu}_{TV} - \psirep(y)\abs{\cO}.
  \end{displaymath}
\end{lem}
\begin{proof}
  For $\mu\in\cM(\cO)$, denote by $\mu=\mu^{a}+\mu^{s}$ the Lebesgue
  decomposition of $\mu$ with respect to Lebesgue measure, and let
  $h=\frac{\d\mu^{a}}{\d x}$ be the Radon-Nikodym derivative of $\mu^a$. As
  $\psirep_\infty(\mu^s)$ is singular by
  \cite[Theorem 4.2]{Temam-book}, we can use the decomposition
  \eqref{eq:82} to obtain
  \begin{equation}\label{eq:56}
    \norm{\psirep(\mu)}_{TV} = \int_\cO \psirep(h)\,\d x + \norm{\psirep_\infty(\mu^s)}_{TV}.
  \end{equation}
  We now estimate the summands separately. For the absolutely continuous
  part we obtain using Assumption \ref{assumptions} \ref{item:9}
  \begin{displaymath}
    \int_\cO \psirep(h)\, \d x
    \geq \frac{\psirep(y)}{y} \int_\cO \abs{h} \d x - \psirep(y)\abs{\cO}
    = \frac{\psirep(y)}{y}\tvnorm{\mu^a} - \psirep(y)\abs{\cO}.
  \end{displaymath}
  
  For the singular part, we note by Lemma \ref{char-rec-fn} that for
  $v\in\cC_c^0(\cO)$ being in $\cD_{\psirep_\infty}$ is equivalent to
  $-\psirep_\infty(1)\leq v \leq \psirep_\infty(1)$, and for such $v$,
  $\psi_\infty^*(v) \equiv 0$. Thus, we get with Corollary \ref{finf1} with
  $k:= \frac{\psirep(y)}{y}$
  \begin{displaymath}
    \int_\cO \psirep_\infty(\mu^s) 
    = \sup_{v\in\cD_{\psirep_\infty}} \left(\int_\cO v\ \d\mu^s - \int
        \psirep_\infty^*(v)\d x\right) \geq \sup_{\substack{v\in \cC_c^0(\cO)\\-k \leq v\leq k}} \int_\cO v\
    \d\mu^s = k \tvnorm{\mu^s}.
  \end{displaymath}
  Thus, we can continue \eqref{eq:56} by
  \begin{displaymath}
    \tvnorm{\psirep(\mu)}
    \geq \frac{\psirep(y)}{y}\tvnorm{\mu^a} + k\tvnorm{\mu^s} -
    \psirep(y)\abs{\cO} = \frac{\psirep(y)}{y}\tvnorm{\mu} - \psirep(y)\abs{\cO},
  \end{displaymath}
  as required.
\end{proof}

\begin{prop}\label{conv+lsc}
  In both settings of Definition \ref{def-energy-fnal},
  $\vphi:\Hd\to[0,\infty]$ is convex and lower-se\-mi\-continuous.
\end{prop}
\begin{proof}
  In the superlinear case, \ie Definition \ref{def-energy-fnal}
  \ref{item:1} applies,
  convexity and lower-semicontinuity of $\vphi$ are proved in \cite[p.\
  68]{Barbu}. In the sublinear case, \ie Definition \ref{def-energy-fnal}
  \ref{item:2} applies, convexity becomes clear by Remark
  \ref{Representation-remark}. It remains to prove lower-semicontinuity in
  the sublinear case.

  \textit{Step 1:} As a preparatory step, we establish weak{*}
  lower-semicontinuity of the functional $\tilde{\vphi}: \cM(\cO)\to[0,\infty)$,
  \begin{displaymath}
    \tilde{\vphi}(\mu)=\norm{\psirep(\mu)}_{TV},
  \end{displaymath}
  for which we have 
  \begin{displaymath}
    \tilde{\vphi}|_{\cM(\cO)\cap\Hd}=\vphi.
  \end{displaymath}
  Consider $\mu_n \to \mu$ weakly{*} for $n\to\infty$. We can assume that
  $\psirep(\mu_n)$ contains a subsequence which is bounded in TV norm
  (otherwise there is nothing to show). Then we select a subsequence
  $(\mu_{n_k})_{k\in\N}$ such that
  $\norm{\psirep(\mu_{n_k})}_{TV} \to \liminf_{n\to\infty}
  \norm{\psirep(\mu_n)}_{TV}$ for $k\to\infty$, from which we can choose a
  nonrelabeled subsequence $(\psirep(\mu_{n_k}))_{k\in\N}$ which converges
  weakly{*} to some $\nu\in\cM(\cO)$ (\eg by \cite[Satz 6.5]{Alt}). By
  \cite[Lemma 2.1]{Temam-Demengel}, we get that
  \begin{displaymath}
    \dup{\cM(\cO)}{\psirep(\mu)}{\eta}{\cC_0^0(\cO)} \leq \dup{\cM(\cO)}{\nu}{\eta}{\cC_0^0(\cO)} = \lim_{k\to\infty}\dup{\cM(\cO)}{\psirep(\mu_{n_k})}{\eta}{\cC_0^0(\cO)} \leq
    \lim_{k\to\infty}\tvnorm{\psirep(\mu_{n_k})}\norm{\eta}_{\cC_0^0(\cO)}
  \end{displaymath}
  for $\eta \in\cC_c^0(\cO)$, $\eta \geq 0$. Now, using that
  $\psirep(\rho)$ is a positive measure for any $\rho\in\cM(\cO)$ by
  \eqref{eq:81}, we obtain
  \begin{align*}
    \tvnorm{\psirep(\mu)}
    = \sup_{\substack{\eta\in\cC_c^0(\cO)\\ \eta\in[0,1]}}\dup{\cM(\cO)}{\psirep(\mu)}{\eta}{\cC_0^0(\cO)}
    &\leq \sup_{\substack{\eta\in\cC_c^0(\cO)\\ \eta\in[0,1]}}
    \lim_{k\to\infty}\dup{\cM(\cO)}{\psirep(\mu_{n_k})}{\eta}{\cC_0^0(\cO)}\\
    &\leq \sup_{\substack{\eta\in\cC_c^0(\cO)\\ \eta\in[0,1]}}\lim_{k\to\infty}\tvnorm{\psirep(\mu_{n_k})} = \liminf_{n\to\infty} \tvnorm{\psirep(\mu_n)},
  \end{align*}
  as required.

  \textit{Step 2:} Assume now that $(u_n)_{n\in\N}\subset\Hd$, $u\in\Hd$, and
  $u_n\to u$ for $n\to\infty$. Being the only non-trivial case, we can
  assume that $(u_n)_{n\in\N}$ contains a subsequence (which we call again
  $(u_n)$) for which $(\vphi(u_{n}))_{n\in\N}$ is bounded. Thus, there are
  measures $\mu_{n}\in\cM(\cO)\cap\Hd$ such that
  \begin{displaymath}
    u_{n}(\eta)=\int_\cO \eta\ \d\mu_{n}\quad\text{for all }\eta\in \cC_c^1(\cO).
  \end{displaymath}
  By definition of $\vphi$, $\vphi(u_n) = \norm{\psirep(\mu_n)}_{TV}$, such that
  Lemma \ref{varphi-TV} implies that $\norm{\mu_{n}}_{TV}$ is bounded.
  Thus, there is $\tilde{\mu}\in\cM(\cO)$ an again nonrelabeled subsubsequence
  $(\mu_{n})_{n\in\N}$ such that $\mu_n \tows \tilde{\mu}$. For
  $\eta\in \cC_c^1(\cO)\subseteq \cC_c^0(\cO)$ we have
  \begin{displaymath}
    \int_\cO \eta\ \d\tilde{\mu}=\lim_{n\to\infty}\int_\cO \eta\
    \d\mu_{n}=\lim_{n\to\infty} u_{n}(\eta)=u(\eta) \leq \norm{u}_\Hd \norm{\eta}_{\Hsob(\cO)},
  \end{displaymath}
  so $\tilde\mu \in \cM(\cO)\cap\Hd$ and $u=\tilde{\mu}$. Using the weak{*} lower-semicontinuity
  of $\tilde{\vphi}$ from Step 1, we get
  \begin{equation}\label{eq:83}
    \vphi(u)=\tilde{\vphi}(\tilde\mu)\leq\liminf_{n\to\infty}\tilde{\vphi}(\mu_{n})=\liminf_{n\to\infty}\vphi(u_{n}).
  \end{equation}
  As this argument works for any bounded subsequence of $(u_n)_{n\in\N}$,
  \eqref{eq:83} is also true for the original sequence $(u_n)_{n\in\N}$.
\end{proof}

As one can see from the definition of the energy functional $\vphi$ in the
second part of Definition \ref{def-energy-fnal}, it has an explicit
representation on $\Hd\setminus\cM(\cO)$, where it is $\infty$, and on
$L^1(\cO)\cap\Hd$, where it is an integral. However, whenever we evaluate
$\vphi$ for general measures in $\cM(\cO)\cap\Hd$, \eg in the uniqueness part of
the proof of Theorem \ref{main-thm}, we need an approximation reducing it
to evaluations on $L^1(\cO)$ functions. This will be made precise in the
following theorem, the proof of which will take the rest of this
section.
\begin{thm}\label{approx-thm}
  Assume that $\psirep$ satisfies \eqref{eq:20} as well as Assumptions
  \ref{assumptions} \ref{item:11}, \ref{item:9}. Let $\vphi$ be defined as
  in Definition \ref{def-energy-fnal} \ref{item:2} and
  $u\in\cM(\cO)\cap\Hd$. Then there exists a sequence $u_n\in L^2(\cO)$
  such that
  \begin{align}
    \label{eq:32}&u_n \tow u\quad\text{in } \Hd,\text{ and }\\
  \label{eq:37}  &\vphi(u_n) \to \vphi(u)
  \end{align}
  for $n\to\infty$.
\end{thm}

\begin{Cor}
  Since convex functions on a real Hilbert space are lower-semicontinuous
  if and only if they are weakly sequentially lower-semicontinuous (see \eg
  \cite[Theorem 9.1]{Bauschke}), Theorem \ref{approx-thm} implies that
  $\vphi$ is the lower-semicontinuous hull of $\vphi|_{L^2(\cO)}$ in $\Hd$,
  which means that
  \begin{equation}\label{eq:93}
    \vphi = \sup\left\{ \confn: \Hd\to[0,\infty]\ \big|\ \confn \text{ convex
        and lower-semicontinuous, } \confn|_{L^2(\cO)} \leq \vphi|_{L^2(\cO)}\right\},
  \end{equation}
  where $\sup$ denotes the pointwise supremum.
\end{Cor}

We will approach Theorem \ref{approx-thm} by giving an explicit
construction for the sequence $(\mu_n)_{n\in\N}$, inspired by the
construction in \cite[Lemma A6.7]{Alt}. It will rely on
applying the original functional to modified functions, which is why we
first introduce several modifications to functions on $\cO$.

We next introduce further notation and recall some concepts relying on
the regularity of the boundary.
\begin{Not}\label{notation-app}
  Since the domain $\cO$ is bounded and smooth, its boundary is locally the
  graph of a smooth function. More precisely, we recall from \cite[Section
  A6.2]{Alt} that for each $y\in\partial\cO$
  there is a neighbourhood $\tilde{U}\subset\R^d$, an orthonormal system
  $e_1,\dots, e_d$ of $\R^d$, $r,h\in\R$ with $r>h>0$, and a smooth bounded
  function $g: \R^{d-1}\to\R$, such that with the notation
  \begin{displaymath}
    x_{,d} := (x_1,\dots, x_{d-1}), \quad \text{for
    }x=\sum_{i=1}^d x_ie_i,
  \end{displaymath}
  we have
  \begin{displaymath}
    \tilde{U} = \left\{ x\in\R^d: \abs{x_{,d} - y} < r \text{ and }\abs{x_d -
        g(x_{,d})} < h\right\},
  \end{displaymath}
  and for $x\in \tilde{U}$
  \begin{align*}
   x_d = g(x_{,d}) \quad&\text{ if and only if } x\in\partial\cO,\\
   x_d \in (g(x_{,d}), g(x_{,d}) + h)\quad &\text{ if and only if }
                                       x\in\cO, \text{ and}\\
   x_d \in (g(x_{,d}) - h, g(x_{,d}))\quad &\text{ if and only if } x\notin\cO.
  \end{align*}
  For technical reasons we set
  \begin{equation}\label{eq:38}
    U = \left\{x\in\tilde{U}: \abs{x_{,d}- y}<\frac{r}{2} \text{ and }\abs{x_d -
        g(x_{,d})} < \frac{h}{2} \right\}.
  \end{equation}
  The boundary $\partial\cO$ is covered by those open sets $U$ belonging to
  all possible reference points $y$. As $\partial\cO$ is compact, we can
  choose a finite subcovering $(U^j)_{j=1}^l$, and for each $U^j$, we
  denote the elements belonging to it by a superindex $j$, e.\ g.\
  $y^j, e_d^j, g^j, h^j, \tilde{U}^j$. At last, we fix an open set $U^0$ with
  $\overline{U^0}\subset \cO$, such that
  $\overline{\cO}\subset \cup_{j=0}^l U^j$ and we set $e_d^0 := 0$.

  Subordinate to the covering $\cup_{j=0}^l U^j$, let now
  $\zeta^0, \dots, \zeta^l$ be a partition of unity on $\overline{\cO}$,
  i.\ e.\
  $0\leq\zeta^j\leq 1, \zeta^j\in \cC^\infty_c(\R^d),
  \supp(\zeta^j)\subseteq U^j$ for all $j=0,\dots,l$, and
  \begin{displaymath}
    \sum_{j=0}^l \zeta^j = 1\quad \text{on }\overline\cO.
  \end{displaymath}
  For $\eta: \cO\to\R$ and $\mu\in \cM(\cO)$, we define $\ext{\eta}:
  \R^d\to\R$ and $\ext{\mu}\in \cM(\R^d)$ as the extended function (resp. measure)
  by zero. Finally, we define for $\rho\in C_c^\infty(\R^d)$ with
  \begin{equation}\label{eq:86}
    \supp(\rho)\subseteq B_1(0),\, \int_\cO \rho\, \d x=1,\ \, \rho(x) =
\rho(-x)
\end{equation}
a Dirac sequence $(\rho_\delta)_{\delta>0}\subset C_c^\infty(\R^d)$ of mollifiers by
\begin{equation}\label{eq:34}
  \rho_\delta(x)=\frac{1}{\delta^d}\,\rho\left(\frac{x}{\delta}\right).
\end{equation}
For $\eta\in L^2(\R^d), \mu\in\cM(\R^d)$, we then define functions
$\rho_\delta\ast\eta$, $\rho_\delta\ast\mu \in C^\infty(\R^d)$
by
\begin{align*}
  \rho_\delta\ast \eta(x)
  = \int_{\R^d}\rho_\delta(x-y) \eta(y)\, \d y \quad\text{and}\quad \rho_\delta\ast\mu(x) = \int_{\R^d}\rho_\delta(x-y)\d\mu(y).
\end{align*}
For brevity, we write $\rho_\delta\ast\eta := \rho_\delta\ast\ext{\eta}$
for $\eta\in L^2(\cO)$.
\end{Not}

The following construction allows to shift a function ``away from the boundary''.
\begin{Def}
  Let $\eps>0$ and $\eta: \cO\to\R$. Then we define $\eta_\eps: \cO\to\R$ by
  \begin{equation}
    \eta_\eps(x) = \sum_{j=0}^l \zeta^j(x)\ext{\eta}(x - \eps e_d^j),
  \end{equation}
  where we recall that $e_d^0$ is set to $0$.
\end{Def}

\begin{Rem} \label{stripe-rem}
  By this construction, we achieve that $\eta_\eps = 0$ on a
  $w(\eps)$-neighbourhood of $\partial\cO$ with
  \begin{equation}\label{eq:84}
    w(\eps) := \min\left\{\dist(U^0,\cO^c),
      \min_{j=1,\dots,l} \left(\min\left\{ \frac{\eps}{2}, \frac{\eps}{2L^j},
          \frac{h^j}{4}, \frac{h^j}{4L^j}\right\}\right)\right\} > 0,
  \end{equation}
  where $L^j$ denotes the Lipschitz constant of $g^j$ defined in Notations
  \ref{notation-app}.
\end{Rem}

\begin{proof}
  The number $w(\eps)$ is obviously strictly positive by the construction
  of the covering $(U^j)_{j=0}^l$. To show the support property, let $j\in\{0,1,\dots,l\}$ and
  $U^j_\eps := U^j\cap ((U^j\cap\cO)+\eps e_d^j)$. By definition,
  $\ext{\eta}(x - \eps e_d^j)=0$ if $x\in U^j\setminus U^j_\eps$. By the
  definition of $\zeta^j$, we furthermore conclude that
  $\zeta^j(x)\ext{\eta}(x - \eps e_d^j) = 0$ for $x \notin
  U^j_\eps$. Consequently,
  \begin{displaymath}
    \eta_\eps: x\mapsto \sum_{j=0}^l \zeta^j(x)\ext{\eta}(x-\eps e_d^j)
  \end{displaymath}
  is supported on
  \begin{displaymath}
    U_\eps := \bigcup_{j=0}^l U^j_\eps,
  \end{displaymath}
  such that it remains to show that $\dist(U_\eps, \cO^c)\geq w(\eps)$, or
  equivalently, that $\dist(U_\eps^j, \cO^c)\geq w(\eps)$ for all
  $j\in\{0,\dots,l\}$.

  For $j=0$, this is trivial by construction of $U_\eps^0 = U^0$ and
  $w(\eps)$. For $j=1,\dots, l$, using the
  coordinate system $(x_{,d}^j, x_d^j)$ we can rewrite
  \begin{displaymath}
    U^j_\eps = \{x\in U^j: x_d^j > g^j(x_{,d}^j) + \eps\}.
  \end{displaymath}
  Hence, we can compute for any $x\in U^j_\eps$, \ie
  $x= \left(x^j_{,d}\, ,\, g^j(x^j_{,d}) + \eps'\right)$ for some $\eps'
  \in (\eps, \frac{h^j}{2})$,
  and $y\in \partial\cO \cap {\tilde{U}}^j$
  \begin{align*}
    \norm{x-y}^2
    &= \norm{x_{,d}- y_{,d}}^2 + \abs{g(x_{,d}) + \eps' - g(y_{,d})}^2\\
    &\geq \norm{x_{,d}- y_{,d}}^2 + (\eps' - \abs{g(x_{,d}) - g(y_{,d})})^2,
  \end{align*}
  where $\norm{\cdot}$ denotes the Euclidean norm both in $\R^d$ and in
  $\R^{d-1}$.
  Letting $L^j$ be the Lipschitz constant of $g^j$, we can then argue that
  either $\norm{x_{,d}- y_{,d}} > \frac{\eps}{2L^j}$ or
  \begin{displaymath}
    \abs{g(x_{,d}) - g(y_{,d})} \leq L^j \frac{\eps}{2L^j} = \frac{\eps}{2},
  \end{displaymath}
  such that $\dist(U^j_\eps, \partial\cO\cap \tilde{U}^j)$ is at least
  $\min\left\{\frac{\eps}{2}, \frac{\eps}{2L^j}\right\}$. By similar
  arguments, we can obtain from the construction of $U^j$ in \eqref{eq:38}
  (note that $r^j>h^j$ by construction) that
  \begin{displaymath}
    \dist(U^j_\eps, (\tilde{U}^j)^c) \geq \min\left\{\frac{h^j}{4}, \frac{h^j}{4L^j}\right\},
  \end{displaymath}
  such that we conclude
  \begin{align*}
    \dist(U^j_\eps, \partial\cO)
    &= \min\{\dist(U^j_\eps, \partial\cO\cap\tilde{U}^j),
      \dist(U^j_\eps, \partial\cO\cap(\tilde{U}^j)^c)\}\\
    &\geq \min\{\dist(U^j_\eps, \partial\cO\cap\tilde{U}^j),
      \dist(U^j_\eps, (\tilde{U}^j)^c)\}\\
    &\geq \min\left\{\frac{\eps}{2}, \frac{\eps}{2L^j}, \frac{h^j}{4},
      \frac{h^j}{4L^j}\right\} \geq w(\eps).
  \end{align*}
\end{proof}

This allows to define the following approximating objects for
$u\in\cM(\cO)\cap\Hd$.
\begin{Def}\label{Def-approx-fnals}
  Let $\eps>0$, $0<\delta\leq \frac{w(\eps)}{2}$ and
  $u\in \cM(\cO)\cap\Hd$. We then define for $\eta\in{\Hsob(\cO)}$
  \begin{align}\label{eq:70}
    \begin{split}
      \tilde{u}_\eps(\eta)
      &= \dup{\Hd}{u}{\eta_\eps}{{\Hsob(\cO)}}\\
      \text{and}\quad\tilde{u}_{\eps,\delta}(\eta)
      &=\dup{\Hd}{u}{\rho_\delta\ast\eta_\eps}{{\Hsob(\cO)}}.
    \end{split}
  \end{align}
  These functionals are in $\Hd$ by Lemma \ref{boundedness-eps} and Lemma
  \ref{boundedness-eps-delta} below. For $\eta\in \cC_0^0(\cO)$, we define
  \begin{equation}\label{eq:26}
    \begin{split}
      \mu_\eps(\eta)
      &= \dup{\cM(\cO)}{\mu}{\eta_\eps}{\cC_0^0(\cO)}\\
      \text{and}\quad\mu_{\eps,\delta}(\eta)
      &=\dup{\cM(\cO)}{\mu}{\rho_\delta\ast\eta_\eps}{\cC_0^0(\cO)}.
    \end{split}
  \end{equation}
  These functionals are in $\cM(\cO)$ by Lemma \ref{approx-meas-bdd}
  below. By uniqueness of the linear
  continuation, this allows to conclude that
  \begin{displaymath}
    u_\eps, u_{\eps,\delta}\in\cM\cap\Hd, \quad\text{as well as}\quad u_\eps=
    \tilde{u}_\eps\text{ and } u_{\eps,\delta} = \tilde{u}_{\eps,\delta}.
\end{displaymath}
\end{Def}

\begin{lem} \label{boundedness-eps}
  Let $\eps>0$ and $\eta\in{\Hsob(\cO)}$. Then the map
  $\Hsob(\cO)\ni\eta\mapsto\eta_\eps\in\Hsob(\cO)$ is linear, and
  \begin{displaymath}
    \norm{\eta_\eps}_{\Hsob(\cO)} \leq C \norm{\eta}_{\Hsob(\cO)},
  \end{displaymath}
  where $C$ only depends on the localising functions $(\zeta^j)_{j=0}^l$,
  the number of covering sets $l$, the Poincaré constant of the domain
  $\cO$ and the spatial dimension $d$.
\end{lem}

\begin{proof}
  The proof of the linearity claim is straightforward and therefore
  skipped. In order to prove boundedness, let
  $V^j$ = $U^j\cap \cO$ and $U^j_\eps := U^j\cap ((U^j\cap\cO)+\eps e_d^j)$
  as before. We first note
  \begin{equation}\label{eq:39}
    \norm{\eta_\eps}_{\Hsob(\cO)} = \norm{\sum_{j=0}^l \zeta^j
      \eta_\eps^j}_{\Hsob(\cO)} \leq \sum_{j=0}^l \norm{\zeta^j\eta_\eps^j}_{\Hsob(\cO)},
  \end{equation}
  where we have written
  \begin{displaymath}
    \eta_\eps^j\in H^1(\R^d), \quad \eta_\eps^j(x) = \ext{\eta}(x - \eps e^j_d).
  \end{displaymath} 
  We now analyse the summands separately, where we make use of the fact
  that for all $j\in \{1,\dots,l\}$, $\zeta^j\in\cC^\infty_c(U^j)$ and
  $\zeta^j\eta^j_\eps$ is supported on $V^j$. In the following,
  $(\partial_i)_{i=1}^d$ represent the weak partial derivatives of first
  order. We then compute for $i\in\{1,\dots,d\}$
  \begin{align*}
    \norm{\partial_i(\zeta^j \eta_\eps^j)}_{L^2(\cO)}
    =& \norm{\partial_i(\zeta^j \eta_\eps^j)}_{L^2(V^j)} \leq \norm{(\partial_i \zeta^j) \eta_\eps^j}_{L^2(V^j)} + \norm{\zeta^j
      \partial_i\eta_\eps^j}_{L^2(V^j)}\\
    &\leq C\norm{\eta_\eps^j}_{L^2(V^j)} + \left(\int_{V^j} \abs{\partial_i(\ext{\eta}(x-\eps
      e_d^j))}^2\d x\right)^{\frac1{2}}\\
    &\leq C\norm{\eta}_{L^2(\cO)} + \left(\int_{U^j_\eps} \abs{(\partial_i\eta)(x - \eps
      e_d^j)}^2\d x\right)^{\frac1{2}}\\
    &\leq C\norm{\eta}_{L^2(\cO)} + \norm{\partial_i\eta}_{L^2(\cO)}.
  \end{align*}
  This yields
  \begin{align*}
    \norm{\zeta^j\eta_\eps^j}_{\Hsob(\cO)}^2
    =& \sum_{i=1}^d \norm{\partial_i(\zeta^j \eta_\eps^j)}_{L^2(\cO)}^2
    \leq \sum_{i=1}^d \left(C\norm{\eta}_{L^2(\cO)} +
      \norm{\partial_i\eta}_{L^2(\cO)}\right)^2\\
    &\leq C\norm{\eta}_{\Hsob(\cO)}^2 + 2\sum_{i=1}^d
      \norm{\partial_i\eta}_{L^2(\cO)}^2 \leq C\norm{\eta}_{\Hsob(\cO)}^2,
  \end{align*}
  where $C$ may depend on $d, \cO$ (through the Poincare constant) and
  $\zeta^j$. Thus, we can continue \eqref{eq:39} by
  \begin{displaymath}
    \norm{\eta_\eps}_{\Hsob(\cO)} \leq \sum_{j=0}^l
    \norm{\zeta^j\eta_\eps^j}_{\Hsob(\cO)} \leq (l+1)C\norm{\eta}_{\Hsob(\cO)},
  \end{displaymath}
  as required.
\end{proof}

Concerning the mollification step, we note that by Remark
\ref{stripe-rem}, $\rho_\delta\ast\eta_\eps(x) = 0$ if
$\dist(x,\partial\cO)\leq \frac{w(\eps)}{2}$ and $0<\delta\leq \frac{w(\eps)}{2}$, so that in this case
we can restrict $\rho_\delta\ast\eta_\eps$ to $\cO$ to get a $\cC_c^1(\cO)$
function. By a slight abuse of notation, we then write
\begin{equation}\label{eq:40}
  (\rho_\delta\ast\eta_\eps)|_\cO = \rho_\delta\ast\eta_\eps \in \cC_c^1(\cO)
  \subseteq {\Hsob(\cO)}\cap\cC_0^0(\cO).
\end{equation}
Also for this step, we have to ensure linearity, which is clear,
and an estimate on the ${\Hsob(\cO)}$ norm, which is done in the following lemma.
\begin{lem} \label{boundedness-eps-delta}
  Let $\eps>0$ and $0<\delta\leq \frac{w(\eps)}{2}$. Then the map
  $\Hsob(\cO)\ni\eta\mapsto\eta_{\eps,\delta}\in\Hsob(\cO)$ is linear, and
  \begin{displaymath}
    \norm{\rho_\delta\ast\eta_\eps}_{\Hsob(\cO)} \leq C \norm{\eta}_{\Hsob(\cO)}\quad \text{for all
    }\eta\in{\Hsob(\cO)},
  \end{displaymath}
  where $C$ is the constant from Lemma \ref{boundedness-eps}.
\end{lem}
\begin{proof}
  The proof of linearity is straightforward. In order to show boundedness, for any $g\in L^2(\cO)$ such that $\rho_\delta\ast g=0$ on $\cO^c$ we can compute
  \begin{align}\label{eq:85}
    \begin{split}
      \norm{\rho_\delta\ast g}_{L^2(\cO)}^2
      =& \int_{\R^d} \left( \int_{\R^d} \rho_\delta(x-y) \ext{g}(y)\,\d
        y\right)^2 \d
      x
      \leq \int_{\R^d} \int_{\R^d} \rho_\delta(x-y) \,(\ext{g}(y))^2\,\d y \,\d x\\
      &= \int_{\R^d} \int_{\R^d} \rho_\delta(x-y) \d x\, (\ext{g}(y))^2\,\d y
      = \norm{\ext{g}}_{L^2(\R^d)}^2 = \norm{g}_{L^2(\cO)}^2,
    \end{split}
  \end{align}
  where in the second step we could apply Jensen's inequality since
  $\rho_\delta(x-y)\,\d y$ is a probability measure for each $x\in\R^d$. By
  Remark \ref{stripe-rem} for all $i\in\{1,\dots,d\}$,
  $\rho_\delta \ast (\partial_i \eta_\eps)$ vanishes outside of $\cO$ if
  $0<\delta\leq \frac{w(\eps)}{2}$. Hence $g$ in \eqref{eq:85} can be
  replaced by each partial derivative $\partial_i\eta_\eps$ which yields
  \begin{align*}
    \norm{\rho_\delta\ast\eta_\eps}_{\Hsob(\cO)}^2
    =& \sum_{i=1}^d\norm{\partial_i(\rho_\delta\ast\eta_\eps)}_{L^2(\cO)}^2
    = \sum_{i=1}^d\norm{\rho_\delta\ast
      \partial_i(\eta_\eps)}_{L^2(\cO)}^2\\
    &\leq \sum_{i=1}^d\norm{\partial_i\eta_\eps}_{L^2(\cO)}^2=
      \norm{\eta_\eps}_{\Hsob(\cO)}^2 \leq C\norm{\eta}_{\Hsob(\cO)}^2,
  \end{align*}
  where the second equality can be found \eg in \cite[Section
  2.23]{Alt} and the last inequality is the statement of Lemma
  \ref{boundedness-eps}.
\end{proof}

\begin{lem}\label{approx-meas-bdd}
  Let $\eps>0$, $0<\delta\leq \frac{w(\eps)}{2}$ and
  $\eta\in\cC_c^0(\cO)$. Then, the map
  \begin{displaymath}
    \cC_c^0(\cO)\ni\eta\mapsto(\eta_\eps, \rho_\delta\ast
  \eta_\eps)\in(\cC_c^0(\cO))^2
  \end{displaymath}
  is linear. Furthermore, we
  have
  \begin{equation}\label{eq:69}
    \norm{\rho_\delta\ast\eta_\eps}_\infty \leq \norm{\eta_\eps}_\infty\leq\norm{\eta}_\infty,
  \end{equation}
  where $\norm{\cdot}_\infty$ denotes the supremum norm.
\end{lem}
  
\begin{proof}
  The proof of the linearity claim is straightforward. In order to show
  boundedness, we first note that for $0<\delta\leq \frac{w(\eps)}{2}$,
  $\rho_\delta\ast\eta_\eps\in \cC_c^0(\cO)$ by construction and Remark
  \ref{stripe-rem}. To obtain \eqref{eq:69}, we estimate for arbitrary
  $x\in\cO$
  \begin{displaymath}
    \abs{\eta_\eps(x)} \leq \sum_{j=0}^l \zeta^j(x) \abs{\ext{\eta}(x-\eps
      e^j_d)} \leq \sum_{j=0}^l \zeta^j(x) \norm{\eta}_\infty = \norm{\eta}_\infty,
  \end{displaymath}
  which yields the second relation. The first one can be seen by
  \begin{displaymath}
    \abs{\rho_\delta\ast\eta_{\eps}(x)}\leq
    \int_{\R^d}\rho_\delta(x-y)\norm{\eta_\eps}_\infty\,\d x = \norm{\eta_\eps}_\infty,
  \end{displaymath}
  which concludes the proof.
\end{proof}

We next analyse how $\vphi$ as given in Definition \ref{def-energy-fnal}
\ref{item:2} acts on the approximating measures from Definition
\ref{Def-approx-fnals}. First, we state that if
$\mu$ is absolutely continuous with respect to the Lebesgue measure, so is
$\mu_\eps$, which we show by giving its density.
\begin{lem}\label{density-mueps}
  Let $\eps>0$, $h\in L^1(\cO)$ and $\mu := h\, \d x \in \cM(\cO)$. Then $\mu_\eps$ has the density
  \begin{displaymath}
    \cO\ni x \mapsto \sum_{j=0}^l \zeta^j(x+\eps e^j_d) \ext{h}(x+\eps e^j_d)
  \end{displaymath}
  with respect to the Lebesgue measure.
\end{lem}

A more direct construction of $\mu_{\eps,\delta}$ is
given by the following lemma.
\begin{lem}\label{sec:alternative-constr}
  Let $\eps,\delta \leq\frac{w(\eps)}{2}$. Then the measure
  $\tmu_{\eps,\delta}\in\cM(\cO)$ defined by
  \begin{equation}\label{eq:72}
    \tmu_{\eps,\delta} := \left( (\rho_\delta \ast \ext{\mu})|_\cO\,\d
      x\right)_\eps,
  \end{equation}
  coincides with $\mu_{\eps,\delta}$.
\end{lem}
The proofs of the preceding lemmas are straightforward and therefore suppressed.

In the rest of this section, we will argue that the sequence
\begin{displaymath}
  \left(\mu_{\frac1{n},\frac1{2}w(\frac1{n})}\right)_{n\in\N}
\end{displaymath}
is an approximation of $\mu\in\cM(\cO)\cap\Hd$ in the sense of Theorem
\ref{approx-thm}. First we address the regularity of $\mu_{\eps,\delta}$,
where $\eps>0$ and
$0<\delta\leq \frac{w(\eps)}{2}$.

\begin{lem}\label{approx-in-l2}
  For all $\eps>0$, $0<\delta\leq \frac{w(\eps)}{2}$, the approximating
  measures $\mu_{\eps,\delta}$ have a bounded density with respect to
  Lebesgue measure.
\end{lem}

\begin{proof}
  The fact that $\mu_{\eps,\delta}$ has a density with respect to Lebesgue
  measure follows from its characterisation in Lemma
  \ref{sec:alternative-constr} and Lemma \ref{density-mueps}. This density
  is bounded in space since
  \begin{displaymath}
    \abs{\sum_{j=0}^l \zeta^j(x+\eps e^j_d) (\rho_{\frac1{2}w(\eps)} \ast
      \ext{\mu})|_\cO(x)} \leq l \sup_{x\in\cO}\abs{\rho_{\frac1{2}w(\eps)}\ast \mu(x)} \leq l
    \sup_{x\in\R^d}\abs{\rho_{\frac1{2}w(\eps)}(x)} \norm{\mu}_{TV}.
  \end{displaymath}
\end{proof}
The first part of the following proposition allows to deduce property
\eqref{eq:32}, while the second part is needed for the further proof of
\eqref{eq:37}.
\begin{prop}\label{conv-prop}
  Let $\rho$ be as in \eqref{eq:34} and $0<\delta_\eps\leq
  \frac{w(\eps)}{2}$.
  \begin{enumerate}
  \item \label{item:14} For $\eta\in{\Hsob(\cO)}$, we have
    \begin{equation}\label{eq:42}
      \rho_{\delta_\eps}\ast\eta_\eps \to \eta\quad \text{for }\eps\to 0 \text{ in }{\Hsob(\cO)}.
    \end{equation}
  \item \label{item:15} For $\eta\in\cC_c^0(\cO)$, we have
    \begin{equation}\label{eq:3}
      \rho_{\delta_\eps}\ast\eta_\eps \to \eta \quad \text{for } \eps\to 0 \text{ in }\cC_c^0(\cO).
    \end{equation}
  \end{enumerate}
\end{prop}
\begin{proof}
  Throughout this proof, we will write $\delta$ instead of $\delta_\eps$, always
  assuming that $0<\delta\leq\frac{w(\eps)}{2}$.
  
  \textit{Proof of part \ref{item:14}:} It is enough to show that for all
  $i\in\{1,\dots,d\}$
  \begin{equation}\label{eq:95}
    \norm{\partial_i(\rho_\delta\ast\eta_\eps) - \partial_i\eta}_{L^2(\cO)}
    \to 0\quad \text{for }\eps\to 0.
  \end{equation}
  By the density of
  $\cC_0^\infty(\cO)$ in ${\Hsob(\cO)}$, for any $\beta>0$ we can choose $\vphi\in \cC_0^\infty(\cO)$ such that
  \begin{equation}\label{eq:43}
    \max\left\{\norm{\vphi-\eta}_{L^2(\cO)}, \norm{\partial_i\vphi - \partial_i\eta}_{L^2(\cO)}\right\} \leq
    \frac{\beta}{6(l+1)\tilde{C}},
  \end{equation}
  where
  \begin{displaymath}
    \tilde{C} := \max\left\{\max_{j=1,\dots,l}(\sup_{\R^d}
    \abs{\partial_i\zeta^j}), 1\right\}
  \end{displaymath}
  As $\ext{\vphi}, \zeta^j\in \cC^1_b(\cO)$ for each $j\in\{1,\dots,l\}$,
  we can choose $\eps_0>0$
  small enough, such that for all $x\in\R^d$ and $y,z\in B_{\eps_0}(x)$
  \begin{equation}\label{eq:44}
    \abs{\partial_i\zeta^j(y)\ext{\vphi}(z) - \partial_i\zeta^j(x)\ext{\vphi}(x)}\leq \frac{\beta}{6(l+1)\abs{\cO}^{\frac1{2}}}
  \end{equation}
  and
  \begin{equation}\label{eq:45}
    \abs{\zeta^j(y)\partial_i\ext{\vphi}(z) - \zeta^j(x)\partial_i\ext{\vphi}(x)}\leq \frac{\beta}{6(l+1)\abs{\cO}^{\frac1{2}}}.
  \end{equation}
  We approach \eqref{eq:95} by splitting the term under consideration into
  the more convenient pieces
  \begin{align*}
    &\norm{\partial_i(\rho_\delta\ast\eta_\eps) - \partial_i\eta}_{L^2(\cO)} =
      \norm{\rho_\delta\ast \partial_i\eta_\eps-\partial_i\ext{\eta}}_{L^2(\R^d)}\\
    &= \norm{\rho_\delta\ast \partial_i(\eta_\eps - \vphi_\eps) + \rho_\delta\ast \partial_i\vphi_\eps - \partial_i\ext{\vphi} + \partial_i\ext{\vphi} - \partial_i\ext{\eta}}_{L^2(\R^d)}\\
    &\leq \norm{\rho_\delta\ast \partial_i(\eta_\eps -
      \vphi_\eps)}_{L^2(\R^d)}
      + \norm{\rho_\delta\ast \partial_i\vphi_\eps -
      \partial_i\ext{\vphi}}_{L^2(\R^d)}
      + \norm{\partial_i\ext{\vphi} - \partial_i\ext{\eta}}_{L^2(\R^d)}\\
    &= \mathrm{(I)} + \mathrm{(II)} + \mathrm{(III)}.
  \end{align*}
  We estimate the summands separately. For the first one we get with the
  convolution estimate (\eg \cite[Section 2.13]{Alt})
  \begin{align}\label{eq:48}
    \begin{split}
      \mathrm{(I)}
      &\leq \norm{\partial_i(\eta_\eps - \vphi_\eps)}_{L^2(\R^d)}\\
      &= \norm{\sum_{j=0}^l \partial_i[\zeta^j (\ext{\eta}(\cdot - \eps e^j_d) -
        \ext{\vphi}(\cdot - \eps e^j_d))]}_{L^2(\R^d)}\\
      &\leq \sum_{j=0}^l \norm{\partial_i\zeta^j\left(\ext{\eta}(\cdot - \eps e^j_d) -
          \ext{\vphi}(\cdot - \eps e^j_d)\right)}_{L^2(\R^d)}\\
      &\qquad + \sum_{j=0}^l \norm{\zeta^j\left(\partial_i\ext{\eta}(\cdot - \eps e^j_d) -
          \partial_i\ext{\vphi}(\cdot - \eps e^j_d)\right)}_{L^2(\R^d)}\\
      &\leq \sum_{j=0}^l \left(\sup_{\R^d}\abs{\partial_i\zeta^j} \norm{\ext{\eta}-
          \ext{\vphi}}_{L^2(\R^d)} +  \norm{\partial_i\ext{\eta} -
          \partial_i\ext{\vphi}}_{L^2(\R^d)}\right) \leq \frac{\beta}{3},
    \end{split}
  \end{align}
  where we used \eqref{eq:43} in the last step. For the second term, we
  recall that $(\zeta^j)_{j=0}^l$ is a partition of unity on the support of
  $\vphi$. Thus, we can compute
  \begin{align*}
    \mathrm{(II)}
    &\leq \sum_{j=0}^l \norm{\rho_\delta \ast \partial_i\left(\zeta^j \ext{\vphi}(\cdot - \eps
      e^j_d)\right) - \partial_i\left(\zeta^j\ext{\vphi}\right)}_{L^2(\R^d)}\\
    &= \sum_{j=0}^l \norm{\rho_\delta\ast \left( \partial_i\zeta^j \ext{\vphi}(\cdot - \eps
      e^j_d) + \zeta^j \partial_i\ext{\vphi}(\cdot - \eps e^j_d)\right) -\partial_i\zeta^j\ext{\vphi} -
      \zeta^j \partial_i\ext{\vphi}}_{L^2(\R^d)}\\
    &\leq \sum_{j=0}^l \norm{\rho_\delta\ast (\partial_i\zeta^j
      \ext{\vphi}(\cdot - \eps e^j_d))
      - \partial_i\zeta^j\ext{\vphi}}_{L^2(\R^d)} + \sum_{j=0}^l\norm{\rho_\delta\ast (\zeta^j
      \partial_i\ext{\vphi}(\cdot-\eps e^j_d)) -
      \zeta^j \partial_i\ext{\vphi}}_{L^2(\R^d)}\\
     &=: \sum_{j=0}^l \mathrm{(IV)}_j + \sum_{j=0}^l\mathrm{(V)}_j.
  \end{align*}
  $\mathrm{(IV)}_j$ and $\mathrm{(V)}_j$ are treated analogously, so we only show
  the estimate for $\mathrm{(V)}_j$, where we choose $\eps<\frac{\eps_0}{2}$
  with $\eps_0$ as for \eqref{eq:44}. Noting that $\rho_\delta$
  integrates to $1$ for any $\delta>0$ and using Jensen's inequality in the
  second step, we obtain
  \begin{align}
    \mathrm{(V)}_j^2
    &= \int_{\R^d} \abs{\int_{\R^d}\rho_\delta(x-y)\left(\zeta^j(y) \partial_i\ext{\vphi}(y - \eps
      e^j_d) - \zeta^j(x) \partial_i\ext{\vphi}(x)\right)\d y}^2\d x\nonumber\\
    &\leq \int_{\R^d}\int_{B_\delta(x)} \rho_\delta(x-y) \abs{\zeta^j(y) \partial_i\ext{\vphi}(y - \eps
      e^j_d) - \zeta^j(x) \partial_i\ext{\vphi}(x)}^2\d y\, \d x.\label{eq:49}
  \end{align}
  As $\partial_i\ext{\vphi}$ is supported on $\cO$ and, for the analogous step for
  $\mathrm{(IV)}$, so is $\ext{\vphi}$, we can argue as in the proof of
  Remark \ref{stripe-rem} to see that the integrand of the outer integral
  is supported on $\cO$. Thus, we can restrict the integration domain to
  obtain
  \begin{align*}
    \eqref{eq:49}
    &= \int_{\cO}\int_{B_\delta(x)} \rho_\delta(x-y) \abs{\zeta^j(y) \partial_i\ext{\vphi}(y - \eps
      e^j_d) - \zeta^j(x) \partial_i\ext{\vphi}(x)}^2\d y\, \d x\\
    &\leq \int_{\cO} \frac{\beta^2}{36\,(l+1)^2\abs{\cO}}\int_{\R^d}
      \rho_\delta(x-y)\,\d y\,\d x = \left(\frac{\beta}{6(l+1)}\right)^2.
  \end{align*}
  While we have used \eqref{eq:45} in the second step, the estimate for
  $\mathrm{(IV)}_j$ uses \eqref{eq:44} instead and gets the same result. We
  conclude
  \begin{equation}\label{eq:46}
    \mathrm{(II)} = \sum_{j=0}^l \left(\mathrm{(IV)}_j +
      \mathrm{(V)}_j\right) \leq \frac{\beta}{3}.
  \end{equation}
  Finally the estimate
  \begin{equation}\label{eq:47}
    \mathrm{(III)} \leq \frac{\beta}{3}
  \end{equation}
  is obvious by property \eqref{eq:43}. Collecting \eqref{eq:48},
  \eqref{eq:46}, and \eqref{eq:47}, we obtain
  \begin{displaymath}
    \norm{\partial_i(\rho_\delta\ast\eta_\eps - \eta)}_{L^2(\cO)} \leq \beta
  \end{displaymath}
  only by choosing $\eps$ small enough and adapting
  $0<\delta\leq \frac{w(\eps)}{2}$, which proves \eqref{eq:42}.

  \textit{Proof of part \ref{item:15}:} Since $\eta$ is now assumed to be
  continuous and to have compact support, it is uniformly continuous. For
  arbitrary $\beta>0$, we can thus fix $\eps_0>0$ such that for all $x,y\in\R^d$
  \begin{displaymath}
    \abs{x-y}\leq\eps_0 \quad\text{implies}\quad \abs{\ext{\eta}(x) - \ext{\eta}(y)} \leq
    \frac{\beta}{l+1}.
  \end{displaymath}
  For $\eps\leq\frac1{2}\eps_0$, we use
  $\delta\leq\frac{w(\eps)}{2}\leq\eps$ by \eqref{eq:84} to calculate for $x\in\cO$ 
  \begin{align*}
    \abs{\rho_\delta\ast\eta_\eps(x) - \eta(x)}
    &= \abs{\int_{B_\delta(x)}
      \rho_\delta(x-y)\left(\sum_{j=0}^l\zeta^j(y)(\ext{\eta}(y-\eps e^j_d) - \eta(x))\right) \d
      y}\\
    &\leq \int_{B_\delta (x)} \rho_\delta(x-y)\sum_{j=0}^l \abs{\ext{\eta} (y - \eps e_d^j) - \ext{\eta}(x)} \d y\\
    &\leq \int_{B_\delta (x)} \rho_\delta(x-y)\sum_{j=0}^l \frac{\beta}{l+1}
      \d y = \beta,
  \end{align*}
  where for the second step we observe that for $y\in B_\delta(x)$,
  we have
  \begin{displaymath}
    \abs{(y -\eps e_d^j) -x} \leq \delta+\eps \leq 2\eps \leq \eps_0. 
  \end{displaymath}
  This proves \eqref{eq:3}.
\end{proof}

We now turn to prove Property \eqref{eq:37}. Recall the definition of a convex
function of a measure from Definition \ref{fofmu}. We need some more lemmas
on measures obtained by this technique, the first of which can be found in
\cite[Equation (2.11)]{Temam-Demengel}.

  \begin{lem}\label{conv-est-meas}
    Let $\psirep$ satisfy \eqref{eq:20} as well as conditions Assumptions
    \ref{assumptions} \ref{item:11},\ref{item:9}. Let $\mu\in\cM(\R^d)$ and
    let $(\rho_\delta)_{\delta>0}$ be a family of mollifying kernels as specified
    in \eqref{eq:86} and \eqref{eq:34}. Then
  \begin{equation}\label{eq:35}
    \int_{\R^d} \psirep(\rho_\delta\ast\mu)\,\d x \leq \int_{\R^d}\psirep(\mu)\quad
    \text{for all }\delta>0.
  \end{equation}
\end{lem}

\begin{Rem}
  Given the assumptions on $\psirep$,
  the theory of Definition \ref{fofmu} indeed also applies to finite
  measures on $\R^d$ (\cf \cite[p. 202]{Temam-book}).
\end{Rem}

\begin{lem} \label{extension-meas}
  Let $\psirep$ satisfy \eqref{eq:20} as well as conditions Assumptions
    \ref{assumptions} \ref{item:11},\ref{item:9}. For $\mu\in\cM(\cO)$ we have
  \begin{equation}\label{eq:36}
    \int_{\R^d}\psirep(\ext{\mu}) = \int_\cO \psirep(\mu).
  \end{equation}
\end{lem}

\begin{proof}
  We define
  \begin{displaymath}
    \cD_1 := \left\{\int_\cO v\,\d\mu -
      \int_\cO \psirep^*(v)\,\d x: v\in L^1(\mu), \psirep^*(v) \in L^1(\cO) \right\}
  \end{displaymath}
  and
  \begin{displaymath}
    \cD_2 := \left\{\int_{\R^d}v\,\d\ext{\mu} -
        \int_{\R^d}\psirep^*(v)\,\d x: v\in L^1(\ext{\mu}), \psirep^*(v) \in L^1(\R^d) \right\},
    \end{displaymath}
    which allows us to write
    \begin{displaymath}
      \int_\cO \psirep(\mu) = \sup \cD_1\quad \text{and}\quad \int_{\R^d}\psirep(\ext{\mu}) = \sup\cD_2.
    \end{displaymath}
  We note that for $v$ satisfying the conditions of $\cD_1$, $\ext{v}$
  satisfies the conditions of $\cD_2$, while the involved integrals
  agree due to the definition of $\ext{\mu}$ and $\psirep^*(0) = 0$. This
  yields ``$\geq$''.

  Conversely, for $v$ satisfying the conditions of $\cD_2$ we can define
  $\tilde{v} = v|_\cO$. $\tilde{v}$ satisfies the conditions of
  $\cD_1$. Furthermore, we have
  \begin{align*}
    \int_\cO \tilde{v}\,\d\mu &= \int_{\R^d}v\,\d\ext{\mu}\quad\text{and}\\
    \int_\cO \psirep^*(\tilde{v})\,\d x &\leq \int_{\R^d} \psirep^*(v)\,\d x\quad\text{due
                                to } \psirep^*\geq 0.
  \end{align*}
  Thus, we have found an element in $\cD_1$ being larger than or equal to
  \begin{displaymath}
    \int_{\R^d}v\,\d\ext{\mu} - \int_{\R^d}\psirep^*(v)\,\d x,
  \end{displaymath}
  which yields ``$\leq$'', completing the proof.
\end{proof}

The key tool to prove the approximation property \eqref{eq:37} is the
following proposition.
\begin{prop}\label{TV-bound-prop}
  Let $\eps>0$ and $0<\delta\leq \frac{w(\eps)}{2}$.
  Then,
  \begin{equation}\label{eq:33}
    \norm{\psirep(\mu_{\eps,\delta})}_{TV} \leq \norm{\psirep(\mu)}_{TV}.
  \end{equation}
\end{prop}
\begin{proof}
  Recall Notations \ref{notation-app} and let $V^j$ =
  $U^j\cap \cO$. Let $(\psiexh_\alpha)_{\alpha>0}\subset C^0_c(\R^d)$ be a sequence of
non-negative cut-off functions compactly supported in $\cO$, which converge
to $1$ pointwise in $\cO$ for $\alpha\to 0$, and each of which is monotonically increasing
on each $V^j$ in $e^j_d$ direction.

Let $h\in L^1(\cO)$ and $\mu = h\,\d x$. In the following argument, we
will need $\psiexh_\alpha(x)\geq \psiexh_\alpha(x-\eps e_d^j)$ for
$x\in V^j$, where $x-\eps e_d^j$ is not a priori in $\cO$. However, since
$\psiexh_\alpha=0$ outside of $\cO$, it is clear that the statement is
valid even if $x-\eps e_d^j\notin\cO$. By the convexity of $\psirep$, the
construction of $(\zeta^j)_{j=0}^l$ and Lemma \ref{density-mueps}, we then
estimate
\begin{align}
  \nonumber\int_\cO \psiexh_\alpha \psirep(\mu_\eps)
  &= \int_\cO \psiexh_\alpha(x)\, \psirep\left(\sum_{j=0}^l \zeta^j(x+\eps e^j_d)
    \ext{h}(x+\eps e^j_d)\right) \d x\\
  \nonumber&\leq \int_\cO \psiexh_\alpha(x) \sum_{j=0}^l \zeta^j(x+\eps e^j_d)
    \psirep(\ext{h}(x+\eps e^j_d))\ \d x\\
  \nonumber&= \int_{\R^d} \psiexh_\alpha(x) \sum_{j=0}^l \zeta^j(x+\eps e^j_d)
    \psirep(\ext{h}(x+\eps e^j_d))\ \d x\\
  \label{eq:96}&= \int_{\R^d}  \psirep(\ext{h}(x))\sum_{j=0}^l \psiexh_\alpha(x-\eps e^j_d) \zeta^j(x) \d x.
\end{align}
We note that $\sum_{j=0}^l \psiexh_\alpha(x-\eps e^j_d) \zeta^j(x)$ is
supported on $\cO$ by Remark \ref{stripe-rem}. Furthermore, by the
construction of $\xi_\alpha$, we have
\begin{displaymath}
  \xi_\alpha(x-\eps e^j_d) \leq \xi_\alpha(x)
\end{displaymath}
for all $x\in V^j$, so this holds especially for $x\in\cO$ for which
$\zeta^j(x) > 0$. Thus, we can continue
\begin{align}\label{eq:100}
  \begin{split}
    \eqref{eq:96} =& \sum_{j=0}^l \int_\cO \psiexh_\alpha(x-\eps e^j_d)
    \zeta^j(x) \psirep(h(x))\,\d x \leq \int_\cO \sum_{j=0}^l\zeta^j(x)\psiexh_\alpha(x) \psirep(h(x))\ \d x\\
    &= \int_\cO \psiexh_\alpha(x) \psirep(h(x))\ \d x = \int_\cO
    \psiexh_\alpha\,\psirep(\mu).
  \end{split}
\end{align}

For a positive Radon measure $\mu$, we have
$\mu(\cO) = \sup\{\mu(K): K\subseteq\cO \text{ compact}\}$. Since any such
$K$ is included in
\begin{displaymath}
  K_\alpha := \{x\in\cO: \mathrm{dist}(x, \cO^c) \geq \alpha\}
\end{displaymath}
for $\alpha$ small enough, we can as well write
$\mu(\cO) = \lim_{\alpha\to 0} \mu(K_\alpha)$. Then, noting that
$\psiexh_\alpha \geq \Ind{K_\alpha}$, we can argue
by definition of the Radon measure of compact sets that
\begin{displaymath}
  \mu(\cO) \geq \int_\cO\psiexh_\alpha\d\mu \geq \mu(K_\alpha) \overset{\alpha\to
    0}{\longrightarrow} \mu(\cO),
\end{displaymath}
thus $\mu(\cO) = \lim_{\alpha\to 0} \int_\cO \psiexh_\alpha \d\mu$.

Hence, we conclude by \eqref{eq:100} for $\mu= h\,\d x$, $h\in L^1(\cO)$, that
\begin{equation}\label{eq:31}
    \int_\cO \psirep(\mu_\eps)= \lim_{\alpha\to 0} \int_\cO \psiexh_\alpha \psirep(\mu_\eps)
    \leq \lim_{\alpha\to 0} \int_\cO \psiexh_\alpha \psirep(\mu) = \int_\cO
    \psirep(\mu).
\end{equation}
Using \eqref{eq:31}, Lemma \ref{conv-est-meas} and Lemma
\ref{extension-meas}, we then obtain for $0<\delta\leq\frac{w(\eps)}{2}$
\begin{align*}
  \int_\cO \psirep(\mu_{\eps,\delta})
  =& \int_\cO \psirep(((\rho_\delta\ast\ext{\mu})|_\cO\,\d x)_\eps) 
  \leq \int_\cO \psirep((\rho_\delta\ast\ext{\mu})|_\cO)\,\d x\\
  &= \int_{\R^d} \psirep(\rho_\delta\ast\ext{\mu})\Ind{\cO}\,\d x
  \leq  \int_{\R^d} \psirep(\rho_\delta\ast\ext{\mu})\,\d
                            x \leq \int_{\R^d} \psirep(\ext{\mu}) = \int_\cO \psirep(\mu),
\end{align*}
which finishes the proof.
\end{proof}

\begin{Cor}\label{bddness-cor}
  Together with Remark \ref{Representation-remark}, Proposition
  \ref{TV-bound-prop} immediately implies
  \begin{displaymath}
    \limsup_{\eps\to 0} \int_\cO \psirep(\mu_{\eps,\delta_\eps}) \leq \int_\cO \psirep(\mu)
  \end{displaymath}
  as long as $0<\delta_\eps \leq \frac{w(\eps)}{2}$.
\end{Cor}

\begin{proof}[Proof of Theorem \ref{approx-thm}.]
  For $\mu$ as in Theorem \ref{approx-thm}, we show that the sequence
  \begin{displaymath}
    (\mu_n)_{n\in\N} := \left(\mu_{\frac1{n},\frac1{2}w(\frac1{n})}\right)_{n\in\N},
  \end{displaymath}
  where $w$ was defined in Remark \ref{stripe-rem}, meets all
  requirements.

  By construction, $\mu_n\in\cM(\cO)\cap\Hd$ for all $n\in\N$, and by Lemma
  \ref{approx-in-l2}, the density of $\mu_n$ is bounded and thus in
  $L^2(\cO)$. Property \eqref{eq:32} is proved in the first part of Proposition
  \ref{conv-prop}. For Property \eqref{eq:37}, note that Corollary
  \ref{bddness-cor} especially shows that $(\mu_n)_{n\in\N}$ is uniformly
  bounded in the TV norm, which means that it contains a subsequence that
  converges weakly{*} to $\psirep(\mu)$ by Proposition \ref{conv-prop} and
  \cite[Lemma 2.1]{Temam-Demengel}. Since this argument can be carried out
  for any subsequence, we get weak{*} convergence for the whole sequence
  and, also by \cite[Lemma 2.1]{Temam-Demengel},
  \begin{displaymath}
    \norm{\psirep\left(\mu_{\frac1{n},\frac1{2}w(\frac1{n})}\right)}_{TV} := \int_\cO
    \psirep\left(\mu_{\frac1{n},\frac1{2}w(\frac1{n})}\right) \to \int_\cO \psirep(\mu) :=
    \norm{\psirep(\mu)}_{TV} \quad \text{as } n\to\infty.
  \end{displaymath}
  This yields \eqref{eq:37} and thereby concludes the proof.
\end{proof}

\section{Proof of the main result} \label{sec:proof-existence}

Throughout this section, we work under Assumptions \ref{assumptions}. We
mostly sketch the argument, which is closely along the lines of \cite[Proof
of Theorem 2.3]{G-R}, and only give more details for the parts where
additional results are needed due to the different
nonlinearity.

We consider the SPDE
\begin{align}\label{eq:8}
  \begin{split}
    \d X^\eps_t &= \eps\Delta X^\eps_t \d t + \Delta\phi^\eps(X^\eps_t)\d
    t + B(t,X^\eps_t)\d W_t,\\
    X_0^\eps &= x_0,
  \end{split}
\end{align}
where we use the notation for the Yosida approximation of \cite[Apppendix C]{G-R} and assume $x_0\in L^2(\Om,\cF_0; L^2)$. Now and in the
following we omit the domain $\cO$ when using Lebesgue and Sobolev spaces.
\begin{lem}
  For all $T>0$, Problem \eqref{eq:8} gives rise to a solution in sense of
  \cite[Definition 4.2.1]{Roeckner} with respect to the Gelfand triple
  $V :=L^2\hookrightarrow \Hd \hookrightarrow (L^2)' = V'$.
\end{lem}
\begin{proof}
  One shows that \eqref{eq:8} fits into the framework of \cite[Example 4.1.11]{Roeckner} with the operator
  \begin{displaymath}
    A(u) = \Delta(\eps u + \phi^\eps(u))\quad \text{for }u\in L^2.
  \end{displaymath}
  The statement then follows by \cite[Theorem 4.2.4]{Roeckner}.
\end{proof}

The following lemma provides an important estimate on the regularity of
these approximate solutions and corresponds to \cite[Lemma B.1]{G-R}:
\begin{lem}\label{Regularity-Lemma}
  Let $\eps>0$, $x_0 \in L^2(\Om, \cF_0; L^2)$ and $T>0$. Then for the
  solution $(X_t^\eps)_{t\in[0,T]}$ to \eqref{eq:8} we have
  \begin{displaymath}
    \E \sup_{t\in[0,T]} \norm{X^\eps_t}_2^2 + \eps \E  \int_0^T
    \norm{X^\eps_r}^2_{H^1_0}\d r \leq C(\E\norm{x_0}_2^2 + 1)
  \end{displaymath}
  with a constant $C>0$ independent of $\eps$.
\end{lem}

\begin{proof}[Proof of Lemma \ref{Regularity-Lemma}]
  Let $(e_i)_{i\in\N}\subset \cC^2_0$ be an orthonormal basis in $\Hd$ of
  smooth eigenvectors to $-\Delta$, and let
  $P^n: \Hd\to H_n := \mathrm{span}\{e_1,\dots,e_n\}$ be the orthogonal
  projection onto the span of the first $n$ eigenvectors. Recall that the
  unique variational solution $X^\eps$ to (\ref{eq:8}) is constructed in
  \cite{Roeckner} as a (weak) limit in $L^2([0,T]\times\Om; L^2)$ of the
  solutions to the Galerkin approximation
  \begin{eqnarray*}
    \d X^n_t &=& \eps P^n \Delta X^n_t \d t + P^n\Delta\phi^\eps(X^n_t)\d
                  t + P^n B(t,X^n_t)\d W^n_t\\
  \nonumber X^n_0 &=& P^nx_0,
  \end{eqnarray*}
  in $H_n$, where for simplicity we omit the $\eps$-dependence of $X^n$,
  and for an orthonormal basis $(g_i)_{i\in\N}$ of $U$ (as defined in
  Assumption \ref{assumptions} \ref{item:16}) we let
  \begin{displaymath}
    W^n_t = \sum_{i=1}^n \sp{J^{-1}(W_t)}{g_i}_U g_i.
  \end{displaymath}
  Using the finite-dimensional Ito formula and the Burkholder-Davis-Gundy
  inequality, one shows the energy estimate
  \begin{displaymath}
    \E\sup_{r\in[0,T]}\norm{X_r^n}_{L^2}^2 +
    \eps\,\E\int_0^T \norm{X_r^n}_\Hsob^2\d r \leq C(\E\norm{x_0}_{L^2}^2
    + 1).
  \end{displaymath}
  Thus, $(X^n)_{n\in\N}$ is bounded in $L^2(\Om; L^\infty([0,T];L^2))$ and in
  $L^2(\Om\times[0,T];\Hsob)$. The latter is a Hilbert space, thus we
  can extract a weakly converging subsequence whose limit can be
  identified with the unique weak $L^2(\Om\times[0,T];L^2)$ limit
  $X^\eps$. The former is the dual space of
  $L^2(\Om; L^1([0,T]; L^2))$ which is separable. Thus, we can extract
  a weak{*} converging subsequence whose limit can again be identified
  with $X^\eps$. By weak (respectively weak{*}) lower-semicontinuity of the norms,
  we can thus pass to the limit $n\to\infty$ to obtain the required
  inequality.
\end{proof}

\begin{proof}[Proof of Theorem \ref{main-thm}.]
  \textit{Existence:} Let $(x_0^n)_{n\in\N}\subset L^2(\Om,\cF_0;L^2)$ such
  that $x_0^n \to x_0$ in $L^2(\Om; \Hd)$, and let
  $X^{\eps_1,n}, X^{\eps_2,n}$ be the solutions to \eqref{eq:8} with
  initial state $x_0^n$ for $\eps_1, \eps_2 > 0$. By the Ito formula on
  $e^{-Kt}\norm{X_t^{\epsi,n} - X_t^{\epsii,m}}_\Hd^2$, the
  Burkholder-Davis-Gundy inequality and \cite[Equation (C.5)]{G-R}, we
  obtain
\begin{align}\label{eq:11}
  \begin{split}
    \E \sup_{t\in[0,T]} \left(e^{-Kt} \norm{X_t^{\epsi,n} -
        X_t^{\epsii,m}}_\Hd^2\right) \leq& \ 2\,\E\norm{x_0^n - x_0^m}_\Hd^2 \\
    &+
    C(\epsi + \epsii)\left(\E\norm{x_0^n}_{L^2}^2 + \E\norm{x_0^m}_{L^2}^2
      + 1\right)
  \end{split}
\end{align}
for $K>0$ large enough. Letting first $\eps\to 0$ and then $n\to\infty$
yields a limit $X \in L^2(\Om; \cC([0,T]; \Hd)$ by completeness, which will
be shown to be an SVI solution. To this end, define with the notation of
\cite[Appendix C]{G-R}
\begin{equation}\label{eq:101}
  \vphi^\eps(v) =
  \begin{cases}
    \int_\cO \psirep^\eps(v) \d x,\quad &v\in L^2,\\
    + \infty, &\text{otherwise,}
  \end{cases}
\end{equation}
for $v\in\Hd$. Using the Ito formula on $e^{-Kt}\norm{X_t^{\eps,n}}_\Hd^2$ and the fact
that $-\Delta \phi^\eps(x) \in \partial\vphi^\eps(x)$ for $x\in\Hsob$, one obtains
\begin{equation}\label{eq:65}
  \E\int_0^t\vphi^\eps(X_r^{\eps,n})\,\d r \leq C + \E\norm{x_0^n}_\Hd^2 \leq
  \tilde{C} 
\end{equation}
for some $C, \tilde{C}>0$ independent of $\eps$ and $n$. Together with
Assumption \ref{assumptions} \ref{item:12}, which allows to use
\cite[Equation (C.4)]{G-R}, and the lower-semicontinuity of $\vphi$ from
Proposition \ref{conv+lsc}, one obtains
part \ref{item:6} of Definition \ref{Def-SVI}.

For the variational inequality part, let $G, Z, t$ be as in Definition
\ref{Def-SVI} \ref{item:7}. Using Ito's formula on $\norm{X^{\eps, n}_t -
  Z_t}_\Hd^2$, \eqref{eq:22} and the weighted Young inequality we obtain
\begin{align}\label{existence-est}
  &\E\norm{X^{\eps, n}_t - Z_t}_\Hd + 2 \E\int_0^t\vphi^\eps(X^{\eps,
  n}_r)\,\d r \nonumber\\
  \leq& \ \E\norm{x_0^n - Z_0}_\Hd^2 + 2\E\int_0^t\vphi^\eps(Z_r)\,\d r\\
  &- 2\E\int_0^t\sp{G_r}{X^{\eps, n}_r -Z_r}_\Hd \d r +
    C\E\int_0^t\norm{X^{\eps, n}_r - Z_r}_\Hd^2\d r\nonumber\\
  &+ 2\E\int_0^t \frac1{2}\eps^{\frac{4}{3}}\norm{\Delta X^{\eps, n}_r}_\Hd^2 +
    \frac1{2}\eps^{\frac{2}{3}}\norm{X^{\eps, n}_r - Z_r}_\Hd^2\d r.\nonumber
\end{align}
Then using \cite[Equations (C.3) and (C.4)]{G-R}, the
lower-semicontinuity of $\vphi$ from Proposition \ref{conv+lsc} and Lemma \ref{Regularity-Lemma}, we can
take first $\liminf_{\eps\to 0}$ and then $\liminf_{n\to\infty}$ to obtain
\eqref{eq:4} and therefore the remaining part \ref{item:7} of Definition
\ref{Def-SVI}.

\textit{Uniqueness:} It remains to show that the solution constructed in the previous
step is unique. To this end, let $x_0,y_0 \in L^2(\Om,\cF_0;\Hd)$,
$(y_0^n)_{n\in\N}\subset L^2(\Om,\cF_0;L^2)$ satisfying $y_0^n \to y_0$ in
$L^2(\Om;\Hd)$ for $n\to\infty$. Let $X$ be an arbitrary SVI solution to
\eqref{eq:1} with initial condition $x_0$ and let
$(Y^{\eps, n})_{\eps>0,n\in\N}$ be the solutions to \eqref{eq:8} with
respective initial conditions $(y_0^n)_{n\in\N}$. One can check that
\begin{equation}\label{eq:75}
Z = Y^{\eps, n}\quad \text{and}\quad G = \eps \Delta
Y^{\eps, n} + \Delta \phi^\eps (Y^{\eps, n})
\end{equation}
are admissible choices
for \eqref{eq:4}. Then, \eqref{eq:4} yields for $t\in[0,T]$
\begin{align}\label{eq:18}
  \begin{split}
    \E\norm{X_t - Y^{\eps,n}_t}_\Hd^2 + 2\E\int_0^t\vphi(X_r)\,\d r
    &\leq \E\norm{x_0 - y_0^n}_\Hd^2 +
    2\E\int_0^t\vphi(Y^{\eps,n}_r)\,\d r\\
    &- 2\E\int_0^t\sp{\eps\Delta Y^{\eps,n}_r + \Delta
      \phi^\eps(Y^{\eps,n}_r)}{X_r - Y^{\eps,n}_r}_\Hd\d r\\
    &+ C\E\int_0^t\norm{X_r - Y^{\eps,n}_r}_\Hd^2\d r.
  \end{split}
\end{align}
For $u\in L^2$ and $\vphi^\eps$ as in \eqref{eq:101}, we have
\begin{equation}\label{eq:15}
  \sp{-\Delta\phi^\eps(Y^{\eps,n})}{u-Y^{\eps,n}}_\Hd +
  \vphi^\eps(Y^{\eps,n}) \leq \vphi^\eps(u)\quad \d
  t\otimes\P\text{-a.\ e.}
\end{equation}
Since $Y^{\eps,n}\in \Hsob\subset L^2$ $\d
t\otimes\P$-a.\ e.\ we can use \cite[Equation (C.4)]{G-R}
to obtain $\d t\otimes\P$-almost everywhere
\begin{equation}\label{eq:110}
  \abs{\vphi^\eps(Y^{\eps,n}) - \vphi(Y^{\eps,n})} \leq
  C\eps\left(1+ \norm{Y^{\eps,n}}_{L^2}^2\right).
\end{equation}
Using \ref{eq:110} and \cite[Equation (C.3)]{G-R}, we can modify
\eqref{eq:15} to obtain
\begin{equation}\label{eq:16}
  \sp{-\Delta\phi^\eps(Y^{\eps,n})}{u-Y^{\eps,n}}_\Hd +
  \vphi(Y^{\eps,n}) \leq \vphi(u) + C\eps\left(1+ \norm{Y^{\eps,n}}_{L^2}^2\right)\quad \d
  t\otimes\P\text{-a.\,e.}.
\end{equation}
Note that \eqref{eq:16} is trivial if $\vphi(u) = \infty$. Furthermore,
\eqref{eq:16} can be deduced analogously for $u\in L^{m+1}\cap\Hd$ in the
superlinear setting, \ie when $\vphi$ is given by \eqref{eq:76}, with $m$
as in Assumption \ref{assumptions} \ref{item:13}. In the sublinear setting,
\ie $\vphi$ is given by \eqref{eq:77}, and $u\in\cM(\cO)\cap\Hd$, we
consider the approximating sequence $(u_j)_{j\in\N}\subset L^2$ given by
Theorem \ref{approx-thm}, such that \eqref{eq:16} is satisfied for all
$u_j, j\in\N$. We then pass to the limit $j\to\infty$ and notice that by
the construction of
$(u_j)_{j\in\N}$ we have both
$\vphi(u_j) \to \vphi(u)$ and
\begin{eqnarray*}
  \sp{-\Delta\phi^\eps(Y^{\eps,n})}{u_j-Y^{\eps,n}}_\Hd&&\\
  = \dup{\Hsob}{\phi^\eps(Y^{\eps,n})}{u_j-Y^{\eps,n}}{\Hd}
  &\longrightarrow& \dup{\Hsob}{\phi^\eps(Y^{\eps,n})}{u-Y^{\eps,n}}{\Hd}\\
  &&= \sp{-\Delta\phi^\eps(Y^{\eps,n})}{u-Y^{\eps,n}}_\Hd.
\end{eqnarray*}
Consequently, replacing $u$ by $X$ in \eqref{eq:16}, we have in any case
\begin{equation}\label{eq:88}
  \sp{-\Delta\phi^\eps(Y^{\eps,n})}{X-Y^{\eps,n}}_\Hd +
  \vphi(Y^{\eps,n}) \leq \vphi(X) + C\eps\left(1+ \norm{Y^{\eps,n}}_{L^2}^2\right)\quad \d
  t\otimes\P\text{-a.\,e.}.
\end{equation}
Using \eqref{eq:88} and the weighted Young inequality, we can modify
\eqref{eq:18} to obtain for $t\in[0,T]$
\begin{align*}
  \E\norm{X_t - Y^{\eps,n}_t}_\Hd^2
  \leq&\ \E\norm{x_0-y_0^n}_\Hd^2\\
  &+ 2\E\int_0^t\eps^{\frac{4}{3}}\norm{\Delta Y^{\eps,n}_r}_\Hd^2\d r
    + \eps^{\frac{2}{3}}\norm{X_r - Y^{\eps,n}_r}_\Hd^2\d r\\
  &+ C\E\int_0^t\norm{X_r - Y^{\eps,n}_r}_\Hd^2\d r +
    C\eps\E\int_0^t\left(1+ \norm{Y^{\eps,n}_t}_{L^2}^2\right)\d r.
\end{align*}
Taking $\eps\to 0$ and then $n\to\infty$ yields
\begin{equation}\label{eq:89}
  \E\norm{X_t -Y_t}_\Hd^2 \leq \E\norm{x_0-y_0}_\Hd^2 +
  C\E\int_0^t\norm{X_r - Y_r}_\Hd^2\d r\quad\text{for }t\in[0,T].
\end{equation}
where $Y$ is the SVI solution constructed from
$(Y^{\eps,n})$ by the limiting procedure at the beginning of this
proof. Gronwall's inequality then yields $X=Y$ if $x_0=y_0$, and thus
uniqueness of SVI solutions. Then, estimate \eqref{eq:19} follows by
applying Gronwall's inequality to \eqref{eq:89} with different initial
values, which concludes the proof.
\end{proof}

\appendix

\section{Generalities on convex functions}
We collect and prove some statements on convex functions
defined on $\R$.

\begin{lem}\label{diffquot}
  Let $f: \R\to[0,\infty)$ be convex with $f(0)= 0$ and
  $x,y\in\R\setminus\{0\}$ with $x<y$. Then
  \begin{equation}\label{eq:79}
    \frac{f(x)}{x} \leq \frac{f(y)}{y}.
  \end{equation}
  In particular, for $x> 0$ this implies $f(x)\leq f(y)$.
\end{lem}

\begin{proof}
  Note that by convexity, we have for $\lambda\in (0,1)$, $x\in\R$
  \begin{equation}\label{eq:97}
    f(\lambda x) = f(\lambda x + (1-\lambda)0) \leq \lambda f(x) +
    (1-\lambda)f(0) = \lambda f(x).
  \end{equation}
  If $x<0<y$, the statement is obvious by the nonnegativity of
  $f$. If $0<x<y$, we use \eqref{eq:97} with $\lambda = \frac{x}{y}$ to get
  \begin{displaymath}
    \frac{f(x)}{x} = \frac{f(\lambda y)}{\lambda y} \leq \frac{\lambda
      f(y)}{\lambda y} = \frac{f(y)}{y},
  \end{displaymath}
  while for $x<y<0$ we use \eqref{eq:97} with $\lambda := \frac{y}{x}$ to get
  \begin{displaymath}
    \frac{f(y)}{y} = \frac{f(\lambda x)}{\lambda x} \geq \frac{\lambda
      f(x)}{\lambda x} = \frac{f(x)}{x},
  \end{displaymath}
  as required.
\end{proof}

\begin{lem}\label{boundconj}
  Let $\psirep$ satisfy Assumptions \ref{assumptions} and $y>0$. Then, if $\psirep(y) > 0$, we have
  \begin{displaymath}
    \psirep^*(-x) = \psirep^*(x) \leq \psirep(y)\quad \text{for }x\in\left[0,\frac{\psirep(y)}{y}\right],
  \end{displaymath}
  where $\psirep^*$ is defined as in Definition \ref{Def-conj-rec}.
\end{lem}

\begin{proof}
  By Remark \ref{facts-conj-rec}, the last part of Lemma \ref{diffquot} and
  the nonnegativity of $\psirep^*$, it is enough to show
  \begin{equation}\label{eq:52}
    \psirep^*\left(\frac{\psirep(y)}{y}\right) \leq \psirep(y).
  \end{equation}
  To verify \eqref{eq:52}, we distinguish three cases for $y'\in\R$. For
  $y'\geq y$ we have by Lemma \ref{diffquot}
  \begin{displaymath}
  \frac{\psirep(y)}{y} y' - \psirep(y') =
    y'\left(\frac{\psirep(y)}{y} - \frac{\psirep(y')}{y'}\right) \leq 0,
  \end{displaymath}
  for $y'\leq0$ we have by the nonnegativity of $\psi$
  \begin{displaymath}
\frac{\psirep(y)}{y} y' - \psirep(y')
    \leq 0,
  \end{displaymath}
  and for $y'\in (0,y)$ we have
  \begin{displaymath}
\frac{\psirep(y)}{y} y' - \psirep(y') \leq
    \frac{\psirep(y)}{y} y = \psirep(y),
  \end{displaymath}
  which yields the claim.
\end{proof}

\begin{lem}\label{Lem:defs_of_recession}
  Let $\psirep$ satisfy Assumptions \ref{assumptions}. For
  $K = dom(\psirep^*) := \{x\in\R: \psirep^*(x)<\infty\}$ we have
\begin{displaymath}
\sup K = \lim_{t\to\infty}\frac{\psirep(t)}{t}\quad \text{and} \quad
\sup (-K) = \lim_{t\to\infty}\frac{\psirep(-t)}{t}.
\end{displaymath}
\end{lem}
\begin{proof}
  We only prove the first statement, the second then becomes clear by
  symmetry. To this end, note first that the limit is actually a supremum, as $\frac{\psirep(t)}{t}$ is
  increasing (by \eqref{eq:79}). Let now $x\in K$, which means that
  $xt - \psirep(t) \leq c_x < \infty$ and thus
  $\frac{\psirep(t)}{t} \geq x - \frac{c_x}{t}$ for all $t\in[0,\infty)$, which
  yields ``$\leq$'' by letting $t\to\infty$.

  Conversely, we have $\frac{\psirep(t)}{t}\in K$ for $t>0, \psirep(t)>0$ by
  by Lemma \ref{boundconj}. As $\psirep^*(0) = 0$, this is true also if
  $\psirep(t)=0$, thereby proving ``$\geq$''. 
\end{proof}

\begin{Cor}\label{finf1}
  Let $\psirep$ satisfy Assumptions \ref{assumptions}. By Lemmas \ref{boundconj} and \ref{Lem:defs_of_recession}, we have that
  \begin{displaymath}
    \psirep_\infty(1) = \psirep_\infty(-1) \geq \frac{\psirep(y)}{y}
  \end{displaymath}
  for $y>0$ with $\psirep(y) > 0$.
\end{Cor}

\begin{lem} \label{char-rec-fn} Let $\psirep$ satisfy Assumptions
  \ref{assumptions}. For the convex conjugate of the recession function, we
  have
  \begin{displaymath}
    \psirep_\infty^*(x):= (\psirep_\infty)^*(x) = \chi_{[-\psirep_\infty(1),\psirep_\infty(1)]}(x)
  \end{displaymath}
  for $x\in\R$, where for an Interval $I$ we have written
  \begin{displaymath}
    \chi_I(x) =
    \begin{cases}
      0, &\text{ if }x\in I\\
      +\infty, &\text{ else}.
    \end{cases}
  \end{displaymath}
\end{lem}
\begin{proof}
  In the superlinear case, \ie \eqref{eq:67} is satisfied, we have
  $\psirep_\infty = \chi_{\{0\}}$ and thus $\psirep^*_\infty \equiv 0$, as required. In
  the sublinear case, we first note that $\psirep_\infty$ is, by definition,
  positively homogeneous, which by symmetry amounts to absolute
  homogeneity. Thus
\begin{displaymath}
\psirep_\infty(x) =
\psirep_\infty(1)\abs{x},
\end{displaymath}
where $\psirep_\infty(1) > 0$ by Corollary \ref{finf1}, which allows to conclude by 
the definition of the convex conjugate.
\end{proof}

\bibliography{mybooks.bib}

\begin{thebibliography}{10}

\bibitem{Alt}
H.~Alt.
\newblock {\em Lineare Funktionalanalysis: Eine anwendungsorientierte
  Einf{\"u}hrung}.
\newblock Masterclass. Springer Berlin Heidelberg, 2012.

\bibitem{Anzellotti85}
G.~Anzellotti.
\newblock The {E}uler equation for functionals with linear growth.
\newblock {\em Trans. Amer. Math. Soc.}, 290(2):483--501, 1985.

\bibitem{BTW}
P.~Bak, C.~Tang, and K.~Wiesenfeld.
\newblock Self-organized criticality: An explanation of the 1/f noise.
\newblock {\em Phys. Rev. Lett.}, 59:381--384, 1987.

\bibitem{BTW88}
P.~Bak, C.~Tang, and K.~Wiesenfeld.
\newblock Self-organized criticality.
\newblock {\em Phys. Rev. A}, 38:364--374, 1988.

\bibitem{Barbu}
V.~Barbu.
\newblock {\em Nonlinear Differential Equations of Monotone Types in Banach
  Spaces}.
\newblock Springer New York, 2010.

\bibitem{Barbu-Bogachev}
V.~Barbu, V.~I. Bogachev, G.~D. Prato, and M.~Röckner.
\newblock Weak solutions to the stochastic porous media equation via
  {Kolmogorov} equations: The degenerate case.
\newblock {\em Journal of Functional Analysis}, 237(1):54--75, 2006.

\bibitem{BDPrR-Diffusivity}
V.~Barbu, G.~Da~Prato, and M.~Röckner.
\newblock Stochastic nonlinear diffusion equations with singular diffusivity.
\newblock {\em SIAM Journal on Mathematical Analysis}, 41(3):1106--1120, 2009.

\bibitem{BDPrR-SPMEandSOC}
V.~Barbu, G.~D. Prato, and M.~R{\"o}ckner.
\newblock Stochastic porous media equations and self-organized criticality.
\newblock {\em Communications in Mathematical Physics}, 285(3):901--923, Feb
  2009.

\bibitem{BDPrR-existence-nonneg}
V.~Barbu, G.~D. Prato, and M.~Röckner.
\newblock Existence and uniqueness of nonnegative solutions to the stochastic
  porous media equation.
\newblock {\em Indiana University Mathematics Journal}, 57(1):187--211, 2008.

\bibitem{BDPrR-existence-strong}
V.~Barbu, G.~D. Prato, and M.~Röckner.
\newblock Existence of strong solutions for stochastic porous media equation
  under general monotonicity conditions.
\newblock {\em The Annals of Probability}, 37(2):428--452, 2009.

\bibitem{BDPrR-FTE}
V.~Barbu, G.~D. Prato, and M.~Röckner.
\newblock Finite time extinction of solutions to fast diffusion equations
  driven by linear multiplicative noise.
\newblock {\em Journal of Mathematical Analysis and Applications}, 389(1):147
  -- 164, 2012.

\bibitem{BR-operatorial}
V.~Barbu and M.~R\"{o}ckner.
\newblock An operatorial approach to stochastic partial differential equations
  driven by linear multiplicative noise.
\newblock {\em J. Eur. Math. Soc. (JEMS)}, 17(7):1789--1815, 2015.

\bibitem{B-R-SVITVF}
V.~Barbu and M.~Röckner.
\newblock Stochastic variational inequalities and applications to the total
  variation flow perturbed by linear multiplicative noise.
\newblock {\em Archive for Rational Mechanics and Analysis}, 209(3):797--834,
  Sep 2013.

\bibitem{BRR-Rd}
V.~Barbu, M.~Röckner, and F.~Russo.
\newblock Stochastic porous media equations in ${\R}^d$.
\newblock {\em Journal de Mathématiques Pures et Appliquées},
  103(4):1024--1052, 2015.

\bibitem{Bauschke}
H.~Bauschke and P.~Combettes.
\newblock {\em Convex Analysis and Monotone Operator Theory in Hilbert Spaces}.
\newblock CMS Books in Mathematics. Springer International Publishing, 2017.

\bibitem{Bauzet-Vallet-Wittbold15}
C.~Bauzet, G.~Vallet, and P.~Wittbold.
\newblock A degenerate parabolic-hyperbolic {C}auchy problem with a stochastic
  force.
\newblock {\em J. Hyperbolic Differ. Equ.}, 12(3):501--533, 2015.

\bibitem{Biswas-Majee}
I.~H. Biswas and A.~K. Majee.
\newblock Stochastic conservation laws: weak-in-time formulation and strong
  entropy condition.
\newblock {\em J. Funct. Anal.}, 267(7):2199--2252, 2014.

\bibitem{DPrR-weaksolns}
G.~Da~Prato and M.~R{\"o}ckner.
\newblock Weak solutions to stochastic porous media equations.
\newblock {\em Journal of Evolution Equations}, 4(2):249--271, May 2004.

\bibitem{DGG}
K.~Dareiotis, M.~Gerencs\'{e}r, and B.~Gess.
\newblock Entropy solutions for stochastic porous media equations.
\newblock {\em J. Differential Equations}, 266(6):3732--3763, 2019.

\bibitem{G-D-T}
K.~{Dareiotis}, B.~{Gess}, and P.~{Tsatsoulis}.
\newblock {Ergodicity for Stochastic Porous Media Equations}.
\newblock {\em arXiv e-prints}, page arXiv:1907.04605, Jul 2019.

\bibitem{Deb-Hof-Vov}
A.~Debussche, M.~Hofmanov\'{a}, and J.~Vovelle.
\newblock Degenerate parabolic stochastic partial differential equations:
  quasilinear case.
\newblock {\em Ann. Probab.}, 44(3):1916--1955, 2016.

\bibitem{Deb-Vovelle}
A.~Debussche and J.~Vovelle.
\newblock Scalar conservation laws with stochastic forcing.
\newblock {\em J. Funct. Anal.}, 259(4):1014--1042, 2010.

\bibitem{Temam-Demengel}
F.~Demengel and R.~Temam.
\newblock Convex functions of a measure and applications.
\newblock {\em Indiana Univ. Math. J.}, 33:673 -- 709, 1984.

\bibitem{Diaz-G-PhysRevA}
A.~D\'{\i}az-Guilera.
\newblock Noise and dynamics of self-organized critical phenomena.
\newblock {\em Phys. Rev. A}, 45:8551--8558, Jun 1992.

\bibitem{Diaz-G}
A.~Díaz-Guilera.
\newblock Dynamic renormalization group approach to self-organized critical
  phenomena.
\newblock {\em EPL (Europhysics Letters)}, 26:177 -- 182, 1994.

\bibitem{Feng-Nualart}
J.~Feng and D.~Nualart.
\newblock Stochastic scalar conservation laws.
\newblock {\em J. Funct. Anal.}, 255(2):313--373, 2008.

\bibitem{Ferro}
F.~Ferro.
\newblock Integral characterization of functionals defined on spaces of {BV}
  functions.
\newblock {\em Rend. Sem. Mat. Univ. Padova}, 61:177--201 (1980), 1979.

\bibitem{FGS17}
F.~Flandoli, B.~Gess, and M.~Scheutzow.
\newblock Synchronization by noise.
\newblock {\em Probability Theory and Related Fields}, 168(3):511--556, Aug
  2017.

\bibitem{Fonseca}
I.~Fonseca and G.~Leoni.
\newblock {\em Modern Methods in the Calculus of Variations: $L^p$ Spaces}.
\newblock Springer Monographs in Mathematics. Springer New York, 2007.

\bibitem{Friz-Gess}
P.~K. Friz and B.~Gess.
\newblock Stochastic scalar conservation laws driven by rough paths.
\newblock {\em Ann. Inst. H. Poincar\'{e} Anal. Non Lin\'{e}aire},
  33(4):933--963, 2016.

\bibitem{Gassiat-Gess}
P.~Gassiat and B.~Gess.
\newblock Regularization by noise for stochastic {H}amilton-{J}acobi equations.
\newblock {\em Probab. Theory Related Fields}, 173(3-4):1063--1098, 2019.

\bibitem{Gess-strongsolns}
B.~Gess.
\newblock Strong solutions for stochastic partial differential equations of
  gradient type.
\newblock {\em J. Funct. Anal.}, 263(8):2355--2383, 2012.

\bibitem{Gess-RAsingular}
B.~Gess.
\newblock Random attractors for singular stochastic evolution equations.
\newblock {\em Journal of Differential Equations}, 255(3):524 -- 559, 2013.

\bibitem{Gess-FTE}
B.~Gess.
\newblock Finite time extinction for stochastic sign fast diffusion and
  self-organized criticality.
\newblock {\em Communications in Mathematical Physics}, 335(1):309--344, Apr
  2015.

\bibitem{G-H}
B.~Gess and M.~Hofmanov\'{a}.
\newblock Well-posedness and regularity for quasilinear degenerate
  parabolic-hyperbolic {SPDE}.
\newblock {\em Ann. Probab.}, 46(5):2495--2544, 2018.

\bibitem{G-R-plaplace}
B.~Gess and M.~R\"{o}ckner.
\newblock Stochastic variational inequalities and regularity for degenerate
  stochastic partial differential equations.
\newblock {\em Trans. Amer. Math. Soc.}, 369(5):3017--3045, 2017.

\bibitem{G-R}
B.~Gess and M.~Röckner.
\newblock Singular-degenerate multivalued stochastic fast diffusion equations.
\newblock {\em SIAM J. Math. Anal.}, 47:4059 -- 4090, 2015.

\bibitem{GS17}
B.~Gess and P.~E. Souganidis.
\newblock Stochastic non-isotropic degenerate parabolic-hyperbolic equations.
\newblock {\em Stochastic Process. Appl.}, 127(9):2961--3004, 2017.

\bibitem{G-T-stability}
B.~Gess and J.~M. T\"{o}lle.
\newblock Stability of solutions to stochastic partial differential equations.
\newblock {\em J. Differential Equations}, 260(6):4973--5025, 2016.

\bibitem{G-T-ergodicity}
B.~Gess and J.~Tölle.
\newblock Ergodicity and local limits for stochastic local and nonlocal
  ${p}$-laplace equations.
\newblock {\em SIAM Journal on Mathematical Analysis}, 48(6):4094--4125, 2016.

\bibitem{G-T-multivalued}
B.~Gess and J.~M. Tölle.
\newblock Multi-valued, singular stochastic evolution inclusions.
\newblock {\em Journal de Mathématiques Pures et Appliquées}, 101(6):789 --
  827, 2014.

\bibitem{Gutenberg-Richter}
B.~Gutenberg and C.~F. Richter.
\newblock {\em Seismicity Of The Earth And Associated Phenomena}.
\newblock Princeton University Press, 1949.

\bibitem{Gyongy-Pardoux}
I.~Gy\"{o}ngy and E.~Pardoux.
\newblock On the regularization effect of space-time white noise on
  quasi-linear parabolic partial differential equations.
\newblock {\em Probab. Theory Related Fields}, 97(1-2):211--229, 1993.

\bibitem{Kim03}
J.~U. Kim.
\newblock On a stochastic scalar conservation law.
\newblock {\em Indiana Univ. Math. J.}, 52(1):227--256, 2003.

\bibitem{Krylov-Roz}
N.~Krylov and B.~Rozovskii.
\newblock Stochastic evolution equations.
\newblock {\em J. Soviet Math.}, 16 (Transl.):1233 -- 1277, 1981.

\bibitem{LPS13}
P.-L. Lions, B.~Perthame, and P.~E. Souganidis.
\newblock Scalar conservation laws with rough (stochastic) fluxes.
\newblock {\em Stoch. Partial Differ. Equ. Anal. Comput.}, 1(4):664--686, 2013.

\bibitem{Mandelbrot}
B.~Mandelbrot.
\newblock The variation of certain speculative prices.
\newblock {\em The Journal of Business}, 36, 1963.

\bibitem{Neuss-ergodicity}
M.~{Neu{\ss}}.
\newblock {Ergodicity for singular-degenerate porous media equations}.
\newblock {\em arXiv e-prints}, page arXiv:1909.05161, Sep 2019.

\bibitem{Noever}
D.~A. Noever.
\newblock Himalayan sandpiles.
\newblock {\em Phys. Rev. E}, 47:724--725, 1993.

\bibitem{OttoPME}
F.~Otto.
\newblock The geometry of dissipative evolution equations: the porous medium
  equation.
\newblock {\em Comm. Partial Differential Equations}, 26(1-2):101--174, 2001.

\bibitem{Pardoux}
E.~Pardoux.
\newblock \'{E}quations aux d\'{e}riv\'{e}es partielles stochastiques de type
  monotone.
\newblock In {\em S\'{e}minaire sur les \'{E}quations aux {D}\'{e}riv\'{e}es
  {P}artielles (1974--1975), {III}, {E}xp. {N}o. 2}, page~10. Coll\`ege de
  France Paris, 1975.

\bibitem{Arenas}
C.~J. {P{\'e}rez}, {\'A}.~{Corral}, A.~{D{\'\i}az-Guilera}, K.~{Christensen},
  and A.~{Arenas}.
\newblock {On Self-Organized Criticality and Synchronization in Lattice Models
  of Coupled Dynamical Systems}.
\newblock {\em International Journal of Modern Physics B}, 10:1111--1151, 1996.

\bibitem{Roeckner}
C.~Prévot and M.~Röckner.
\newblock {\em A Concise Course on Stochastic Partial Differential Equations}.
\newblock Springer Berlin/Heidelberg, 2007.

\bibitem{RRW}
J.~Ren, M.~Röckner, and F.-Y. Wang.
\newblock Stochastic generalized porous media and fast diffusion equations.
\newblock {\em Journal of Differential Equations}, 238(1):118--152, 2007.

\bibitem{Temam-book}
R.~Temam.
\newblock {\em Mathematical problems in plasticity}.
\newblock Gauthier-Villars, 1985.

\bibitem{VazquezPME}
J.~Vazquez.
\newblock {\em The Porous Medium Equation: Mathematical Theory}.
\newblock Oxford Mathematical Monographs. Clarendon Press, 2006.

\bibitem{Watkins-Pruessner}
N.~W. Watkins, G.~Pruessner, S.~C. Chapman, N.~B. Crosby, and H.~J. Jensen.
\newblock 25 years of self-organized criticality: Concepts and controversies.
\newblock {\em Space Science Reviews}, 198(1):3--44, Jan 2016.

\end{thebibliography}
\bibliographystyle{abbrv}

\begin{flushleft}
\small \normalfont
\textsc{Marius Neuß\\
Max--Planck--Institut f\"ur Mathematik in den Naturwissenschaften\\
04103 Leipzig, Germany}\\
\texttt{\textbf{marius.neuss@mis.mpg.de}}
\end{flushleft}

\end{document}